\documentclass[final,1p,times]{elsarticle}
\usepackage{lineno,hyperref}
\usepackage{amssymb}
\usepackage{mathtools}
\usepackage{amsthm}
\usepackage{amsmath}
\usepackage{algorithm}
\usepackage{caption}  
\usepackage{relsize}
\usepackage{xcolor}
\newcommand\myworries[1]{\textcolor{black}{#1}}
\newcommand\myred[1]{\textcolor{black}{#1}}
\usepackage{enumitem}
\usepackage{algpseudocode}
\usepackage{multirow}
\usepackage{epstopdf}
\newtheorem{theorem}{Theorem}

\newtheorem{corollary}{Corollary}
\modulolinenumbers[5]

\journal{Linear Algebra and its Applications}

\begin{document}

\begin{frontmatter}

\title{Stability Analysis \\ of \\ Bilinear Iterative Rational Krylov Algorithm\tnoteref{mytitlenote}}
\tnotetext[mytitlenote]{This material is based upon work supported by Council of Scientific and Industrial Research (India) Grant Number 25/(0220)/13/EMR-II.}


\author[mymainaddress]{Rajendra Choudhary}
\ead{phd1301201004@iiti.ac.in}

\author[mymainaddress]{Kapil Ahuja\corref{mycorrespondingauthor}}
\cortext[mycorrespondingauthor]{Corresponding author}
\ead{kahuja@iiti.ac.in}
\ead[url]{http://www.iiti.ac.in/people/~kahuja/}
\address[mymainaddress]{Discipline of Computer Science and Engineering, Indian Institute of Technology Indore}

\begin{abstract}
	\par
	Models coming from different physical applications are very large in size. Simulation with such systems is expensive so one usually obtains a reduced model (by model reduction) that replicates the input-output behaviour of the original full model. A recently proposed algorithm for model reduction of bilinear dynamical systems, Bilinear Iterative Rational Krylov Algorithm (BIRKA), does so in a \myworries{locally} optimal way. 
This algorithm requires solving very large linear systems of equations. Usually these systems are solved by direct methods (e.g., LU), which are very expensive. A better choice is iterative methods (e.g., Krylov). However, iterative methods introduce errors in linear solves because they are not exact. They solve the given linear system up to a certain tolerance. We prove that \myworries{under some mild assumptions} BIRKA is stable with respect to the error introduced by the inexact linear solves. \myworries{We also analyze the accuracy of the reduced system obtained from using these inexact solves} and support all our results by numerical experiments.	
\end{abstract}
\end{frontmatter}
\section{Introduction}\label{Sec:Intro}
\par
A dynamical system describes a relation between two or more measurable quantities by a set of differential equations. The system may be linear or nonlinear. A bilinear dynamical system is one such weakly nonlinear system. The system can be described both in the time domain and in the frequency domain. In the time domain, a Multiple Input Multiple Output (MIMO) bilinear dynamical system with  \textit{m} inputs and \textit{p} outputs is represented as follows \cite{Dirac3,Dirac5}: 
	\begin{alignat}{2}\label{eq:1}
	\zeta : \begin{cases}
	\quad	\dot{x}(t) &= Ax(t) +  \sum\limits_{k=1}^{m} N_{k}x(t) u_{k}(t)+Bu(t),  \\
	\quad	y(t) &= Cx(t), 
	\end{cases}
	\end{alignat}
	\myworries{where $ A,\ N_{k}  \in  \mathbb{R}^{n\times n}$ for $k =1,\ \ldots, \ m$, $ B  \in  \mathbb{R}^{n \times m}$ and $C  \in  \mathbb{R}^{p\times n} $. Also, $u(t)=[u_1(t) \ \ldots \ \ u_m(t)]^T \in \mathbb{R}^m$, $y(t): \mathbb{R} \rightarrow \mathbb{R}^p $ and $x(t): \mathbb{R} \rightarrow \mathbb{R}^n$.} 
	  It is not possible to write the transfer function of a complete bilinear dynamical system, therefore, in \cite{Dirac5,Dirac10} the authors represent the bilinear dynamical system in the frequency domain by a series of subsystem transfer functions, i.e., 
	\begin{center}
		$\zeta= \lim\limits_{M\to \infty}\zeta_M$,	
	\end{center}
	where $ \zeta_M = \left \{ H_1(s_1),\ H_2(s_1,\ s_2),\ H_3(s_1,\ s_2,\ s_3),\ \ldots H_M \left ( s_1 ,\ s_2, \ \ldots,\ s_M \right ) \right \}$ and $s_1 ,\ s_2, \ \ldots,\ s_M$ are the frequencies. The transfer function of the $k^{\text{th}}$ order subsystem is given as follows \cite{Dirac5}: 
	\begin{align}\label{transferfunBilinearEq}
	H_k \left ( s_1 , \ s_2,\ \ldots,\ s_k \right ) = & C\left ( s_kI_n - A \right )^{-1}\bar{N} \left [ I_m \otimes \left ( s_{k-1}I_n - A \right )^{-1} \right ]\left ( I_m \otimes \bar{N} \right ) \ldots \notag \\
	&\cdot \left [ \underbrace{I_m \otimes \ldots \otimes I_m}_{k-2\ \textnormal{times}} \otimes \left ( s_{2}I_n - A \right )^{-1} \right ]\left ( \underbrace{I_m \otimes \ldots \otimes I_m}_{k-2\ \textnormal{times}} \otimes \ \bar{N} \right ) \notag\\
	& \cdot \left [ \underbrace{I_m \otimes \ldots \otimes I_m}_{k-1\ \textnormal{times}} \otimes \left ( s_{1}I_n - A \right )^{-1} \right ]\left ( \underbrace{I_m \otimes \ldots \otimes I_m}_{k-1\ \textnormal{times}} \otimes \ B \right ), 
	\end{align}
	where $\bar{N}=[N_1 \ \ldots\ N_m]$; $I_n$ and $I_m$ are the identity matrices of size $n$ and $m$, respectively; and $\otimes$ denotes Kronecker product (defined later).
\par 	 
If in \eqref{eq:1}, the matrix $N$ is a zero matrix, then the system is a linear dynamical system. That is, a MIMO linear dynamical system is represented as
\begin{align}\label{eq:2}
	\begin{split}
		\dot{x}(t) & = Ax(t) + Bu(t), \\
		y(t) & = Cx(t).
	\end{split}
\end{align}
The transfer function of the linear dynamical system in the frequency domain is defined as follows:
\begin{equation}\label{transferfunLinearEq}
H(s) = C(sI_n-A)^{-1}B. 
\end{equation}
\par
In general, dynamical systems corresponding to real world  applications are extremely large in size. Simulation and computation with such systems requires large amount of space and time. By using model reduction techniques \cite{Dirac6}, these large dynamical systems are reduced into a smaller size, which makes the simulation and computation easier. Model reduction can be done in many ways, i.e., by using balanced truncation, Hankel approximations or Krylov projection \cite{Dirac6}. Projection methods obtain the reduced model by projecting the original full model on a lower dimensional subspace, and are quite popular. In literature, there are several techniques of projecting a dynamical system \cite{Dirac6,Dirac4,serkphd,zbai,kahujaphd,Dirac7}. The Petrov-Galerkin projection is one such projection technique that gives nice properties in the reduced model. Interpolation is usually used to obtain the subspaces involved in the Petrov-Galerkin projection.
\par 
Based upon the theory of Petrov-Galerkin based interpolatory model reduction, authors in \cite{Dirac4,inexirka,mimo} have proposed Iterative Rational Krylov Algorithm (IRKA) for model reduction of linear dynamical systems. IRKA provides the reduced model that is optimal (the kind of optimality is discussed in the next section). Similar to IRKA, authors in \cite{Dirac3,Dirac5,doi:10.1137/130947830,tobiasphd} have proposed Bilinear Iterative Rational Krylov Algorithm (BIRKA) for model reduction of bilinear dynamical systems.
\par 
The main computational bottleneck in reducing larger models (or dynamical systems) is solving large sparse linear systems of equations. The reason for this is that typically, model reducers use direct solvers, e.g., LU factorization to solve such linear systems of equations, which are expensive. The solution to this scaling problem is to use iterative methods, e.g., Krylov subspace methods.
\par 
Application of Krylov subspace methods for IRKA has been done \cite{ahuja2012recycling, sarahsms, sarahsphd}. Iterative methods are inexact, i.e., they solve linear systems of equations up to a certain stopping tolerance. Hence, it becomes important to check if the model reduction algorithm (IRKA or BIRKA) is stable with respect to these inexact solves. In other words, we need to check that small errors in linear solves does not substantially deteriorate the quality of the reduced model. For IRKA, stability analysis has been done in \cite{Dirac2}. We do the same for BIRKA, i.e., prove that BIRKA is stable with respect to the inexact linear solves. From this work users will have more confidence in using iterative solvers for BIRKA. 
\par 
In the next section (Section \ref{Sec:PGBIMRF}), we discuss model reduction by a Petrov-Galerkin based interpolatory model reduction framework. We discuss stability of BIRKA in Section \ref{Sec:BackwardStability}. \myworries{In Section \ref{Sec:Analysis}, we analyze invertibility assumptions of all involved matrices as well as the accuracy of the reduced system obtained from a backward stable BIRKA.} We support our theory with numerical experiments in Section \ref{Sec:Experiment}, and give concluding remarks as well as future directions in Section \ref{Conclusions}. For the rest of this paper we use the terms and notations as listed below.
\begin{enumerate}
	\item [a.] In literature \cite{Dirac3}, the $H_2-$norm of a bilinear dynamical system is defined as
	\begin{align}\label{h2normDefEq}
	\left \| \zeta \right \|_{H_2}^{2}  = vec(I_p)^T
	\left ( C\otimes C \right )  \left ( -A \otimes I_n -I_n \otimes A - \sum_{k=1}^{m}N_k \otimes N_k \right )^{-1} \left ( B \otimes B \right )
	vec(I_m), 
	\end{align}
where $ I_p$ is an identity matrix of size $p$. If the type of norm is not written, then in the case of functional norm it is a $H_2 -$norm. In the case of matrices it is a 2-norm.
	\item [b.] The Kronecker product between two matrices $P$ (of size $m\times n$), and $Q$ (of size $s \times t$) is defined as
	\begin{equation}
		 P\otimes Q =\begin{bmatrix}
		p_{11}Q & \cdots  & p_{1n}Q\\ 
		\vdots  & \ddots  & \vdots\\ 
		p_{m1}Q & \cdots & p_{mn}Q 
		\end{bmatrix}, \notag
	\end{equation}
	where $p_{ij}$ is an element of matrix $P$ and order of $P\otimes Q$ is $ms\times nt$. 
	\item [c.] $vec$ operator on a matrix $P$ is defined as
	\begin{equation}
		vec(P)= \left ( p_{11},\ \ldots,\ p_{m1},\ p_{12},\ \ldots,\ p_{m2},\ \ldots \ \ldots,\ p_{1n},\ \ldots,\ p_{mn} \right )^T. \notag
	\end{equation}
	\item [d.] Also, $\mathbb{R}$ denotes the set of real numbers and $\mathbb{F}$ denotes the discrete subset of real numbers.
\end{enumerate}
\renewcommand{\thefootnote}{\fnsymbol{footnote}}
\section{Petrov-Galerkin Based Interpolatory Model Reduction Framework} \label{Sec:PGBIMRF}
\par
According to the Petrov-Galerkin projection, the residual of a dynamical system obtained after projecting on a lower dimensional subspace, is made orthogonal to some other subspace defined by a test basis. Let $\eta_i$ denote the residual of this dynamical system, then according to the Petrov-Galerkin condition, $\eta_i \perp \mathit{L}$, where $\mathit{L}$ denotes any test subspace. 
\par
The subspace on which we project, and the orthogonal subspace are not known to us. We can arbitrarily pick these subspaces, but then we cannot guarantee a good input-output behaviour from the reduced model. For the reduced model to provide a high fidelity approximation to the input-output behaviour of the original full model, we use interpolation to obtain these subspaces.
In \cite{Dirac4}, authors give an algorithm for model reduction of linear dynamical systems called IRKA (Iterative Rational Krylov Algorithm). IRKA is a Petrov-Galerkin based interpolatory model reduction algorithm.
\myworries{For a certain type of linear dynamical systems, IRKA locally converges to a local minimum of the underlying $H_2-$optimization problem \cite{flagg12Convergence}}. For $H_2-$optimality discussion in the linear case we refer the reader to \cite{Dirac4} and \cite{flagg12Convergence}. We discuss $H_2-$optimality in the bilinear case below.
\par 
Next, we apply Petrov-Galerkin based interpolatory model reduction to a bilinear dynamical system. This is a short summary of the original work in \cite{Dirac3} and \cite{Dirac5}.
After reduction, the bilinear system \eqref{eq:1} can be represented as \cite{Dirac3}
\begin{align}\label{eq:5}
	\zeta_r : \begin{cases}
	\quad	\dot{x}_r(t)&= A_r x_r (t) + \sum\limits_{k=1}^{m} N_{k_r} x_r(t) u_k(t)+ B_r u(t), \\
	\quad 	y_r(t) &= C_r x_r(t), 
	\end{cases}
\end{align}
where $A_r,\ N_{k_r}\in \mathbb{R}^{r \times r},\  B_r \in \mathbb{R}^{r\times m} \ \textnormal{and} \ C_r\in \mathbb{R}^{p\times r} \ $ for $k=1, \ \ldots,\ m$ with $r \ll n$. \myworries{We want $\zeta_r$ to approximate $\zeta$ in an appropriate norm, and hence, $y_r(t)$ should be nearly equal to $y(t)$ for all admissible inputs.} Let the two r-dimensional subspaces, $\mathcal{V}_r$ and $\mathcal{W}_r$, be chosen in such a way that $\mathcal{V}_r = Range(V_r) \ \text{and} \ \mathcal{W}_r = Range(W_r)$, where $V_r  \in   \mathbb{R}^{n\times r} \ \text{and} \  W_r  \in  \mathbb{R}^{n\times r}$ are matrices. We project the original full model \eqref{eq:1} to a lower dimensional subspace, i.e., 
$x\left ( t \right )\approx V_r x_r(t)$, and enforce the Petrov-Galerkin condition  \cite{Dirac3,Dirac5} 
\begin{alignat*}{2}
\begin{split}
W_r ^T & \left( V_r \dot{x}_r(t)-A V_r x_r(t) -  \sum_{k=1}^{m} N_k V_r x_r(t) u_k(t)- B u(t) \right) = 0,\\
y(t) & = C V_rx_r(t).
\end{split}	
\end{alignat*} 
Comparing the above equations with \eqref{eq:5}, we get 
\begin{align}\label{redSystemEq}
	A_{r}= \left ( W_r^T V_r   \right )^{-1} W_{r}^{T} A V_{r},\ N_{k_r}= \left ( W_r^T V_r   \right )^{-1} W_{r}^{T} N_k V_{r},\ B_{r}= \left ( W_r^T V_r   \right )^{-1} W_{r}^{T} B, \ \textnormal{and} \  C_{r}= C V_{r},
\end{align}  
where $\left ( W_r^T V_r   \right )$ is assumed to be invertible. Obtaining such an invertible matrix is not hard \cite{Dirac3}. Different selection of the subspaces $\mathcal{V}_r$ and $\mathcal{W}_r$ give different reduced models, but we choose the subspaces $\mathcal{V}_r$ and $\mathcal{W}_r$  by enforcing interpolation. In the case of bilinear systems, there are two ways of doing interpolation \cite{Dirac5}.
\par 
A bilinear system can be represented by a series of subsystem transfer functions. If we apply certain interpolation conditions on a finite number of subsystems then, it is called \textit{subsystem interpolation} \cite{Dirac5}. Another way is \textit{Volterra series interpolation}. \myworries{Here, interpolation is done on a weighted sum of all Volterra kernel transfer functions given by \eqref{transferfunBilinearEq}. 
We refer the reader to \cite{Dirac5, opac-b1090598} for a detailed discussion on the definition of the Volterra series, the Volterra kernels, and the subsequent derivations.}  
\par 
As the subsystem interpolation approach is unable to satisfy any optimality condition \cite{Dirac5} (error between the original full model and the reduced model is minimum in some norm), so our focus is on the Volterra series interpolation.
We need to know how to build $V_r$ and $W_r$ such that the conditions of the Volterra series interpolation are satisfied. We also need to decide where to interpolate so that we get an optimal reduced model. Here, we focus on $H_2-$optimality.
\par 
  In a bilinear system, the following error system expression is differentiated for getting the $H_2 -$ optimality conditions \cite{Dirac3}$:$
\begin{align}
	& \left \| \zeta -\zeta_r \right \|_{H_2}  = \myred{vec(I_{2p})^T}
	\left ( \begin{bmatrix}
		C & -\check{\check{C}}
	\end{bmatrix} \otimes \begin{bmatrix}
	C & -\check{C}
\end{bmatrix} \right ) \times \notag\\ & \left ( -  \begin{bmatrix}
A & 0\\ 
0 & {\Lambda}
\end{bmatrix}   \otimes    \begin{bmatrix}
I_n & 0\\ 
0 & I_r
\end{bmatrix}     -  \begin{bmatrix}
I_n & 0\\ 
0 & I_r
\end{bmatrix}   \otimes    \begin{bmatrix}
A & 0\\ 
0 & \check{A}
\end{bmatrix}    -  \sum_{k=1}^{m}\begin{bmatrix}
N_k & 0\\ 
0 & \check{\check{N}}^T_k
\end{bmatrix}   \otimes    \begin{bmatrix}
N_k & 0\\ 
0 & \check{N}_k
\end{bmatrix}        \right )^{-1}  \notag\\
&\times   
\left ( \begin{bmatrix}
	B\\ 
	\check{\check{B}}^T
\end{bmatrix} \otimes \begin{bmatrix}
B\\ 
\check{B}
\end{bmatrix}  \right ) \myred{vec(I_{2m})}, \label{eq:4}
\end{align}
where $\check{A},\ \check{B},\ \check{C}$ and $\check{N}_k$ are the initial guesses for the reduced system. {Also,} $\check{A}= R \Lambda R^{-1} , \  \check{\check{B}}= \check{B}^T R^{-T}, \check{\check{C}}= \check{C} R \ \textnormal{and} \ \check{\check{N}}_k= R^T \check{N}^T_k R^{-T}$. Performing interpolation on the inverse images of the reduced system poles helps achieve $H_2 -$optimality. Theorem \ref{theorem:InterpolationCond} below summarizes this \myred{where the poles of the transfer function of every reduced subsystem (say $H_{r_k}$)  are computed (say represented by $\lambda_{l_1}, \ \lambda_{l_2}, \ \ldots,\ \lambda_{l_k}$), inverted (leading to $-\lambda_{l_1}, \ -\lambda_{l_2}, \ \ldots,\ -\lambda_{l_k}$), and finally, interpolation is performed at these points.}
\begin{theorem}\cite{Dirac5}
	\label{theorem:InterpolationCond}
	Let $\zeta$ be a bilinear system of order n. Let $\zeta_r$ be an $H_2 -$optimal approximation of order r. Then, $\zeta_r$ satisfies the following multi-point Volterra series interpolation conditions:
	\begin{alignat}{2}
		\sum_{k=1}^{\infty}\sum_{l_1=1}^{r} \ldots \sum_{l_{k}=1}^{r} \phi_{l_1,\ l_2,\ \ldots, \ l_{k}} H_k \left( -\lambda_{l_1},\ -\lambda_{l_2},\ \ldots,\ -\lambda_{l_k} \right) \  = \notag\\ \sum_{k=1}^{\infty}\sum_{l_1=1}^{r} \ldots \sum_{l_{k}=1}^{r} \phi_{l_1,\ l_2,\ \ldots,\ l_{k}}  H_{r_k} \left( -\lambda_{l_1},\ -\lambda_{l_2},\ \ldots,\ -\lambda_{l_k} \right),\quad \textnormal{and} \notag 
	\end{alignat}
	\begin{alignat}{2}
		\sum_{k=1}^{\infty}\sum_{l_1=1}^{r} \ldots \sum_{l_{k}=1}^{r} \phi_{l_1,\ l_2,\ \ldots,\ l_{k}}                \left ( \sum_{j=1}^{k} \frac{\partial }{\partial s_j} H_k\left ( -\lambda_{l_1},\ -\lambda_{l_2},\ \ldots,\ -\lambda_{l_k}\right )  \right ) = \notag \\ \sum_{k=1}^{\infty}\sum_{l_1=1}^{r} \ldots \sum_{l_{k}=1}^{r} \phi_{l_1,\ l_2,\ \ldots,\ l_{k}}    \left (  \sum_{j=1}^{k} \frac{\partial }{\partial s_j}  H_{r_k}\left ( -\lambda_{l_1},\ -\lambda_{l_2},\ \ldots, \ -\lambda_{l_k}\right )    \right ),\notag	
	\end{alignat}
	where $\phi_{l_1,\ l_2,\ \ldots,\ l_{k}}$ and $\lambda_{l_1},\ \lambda_{l_2},\ \ldots,\ \lambda_{l_k}$ 	 are residues and poles of the transfer function $H_{r_k}$ associated with $\zeta_r$, respectively.		
\end{theorem}  
\par
Obtaining the residues and the poles of the $H_2-$optimal reduced model is not possible since we do not have such a system. In \cite{Dirac3, doi:10.1137/130947830} the authors propose Bilinear Iterative Rational Krylov Algorithm (BIRKA),  
%
%
which at convergence, ensures that the conditions of Theorem \ref{theorem:InterpolationCond} are satisfied. BIRKA gives a locally $H_2-$optimal reduced model.
Algorithm \ref{BIRKAAlgo} lists BIRKA.
\begin{algorithm}
	\caption{BIRKA~\cite{Dirac3,Dirac5}}
	\label{BIRKAAlgo}
	\begin{algorithmic}[1]
		\\ Given an input bilinear dynamical system $A,\ N_1, \ \ldots, \ N_m , \ B,\ C$.
		\\ Select an initial guess for the reduced system as $\check{A},\ \check{N}_1, \ \ldots, \ \check{N}_m,\ \check{B},\ \check{C}$. Also select stopping tolerance $btol$.
		\\while $\left ( \textnormal{relative change in eigenvalues of} \ \check{A}  \geq btol \right )$ \begin{enumerate}
			\item[a.] $R \Lambda R^{-1} = \check{A},\ \check{\check{B}}= \check{B}^T R^{-T}, \ \check{\check{C}}=\check{C}R, \ \check{\check{N}}_k= R^T \check{N}_k R^{-T}$ for $k=1,\ \ldots,\ m$. 
			\item[b.] $vec\left ( V \right ) = \left ( -\Lambda \otimes I_n - I_{r} \otimes A - \sum\limits_{k=1}^{m} \check{\check{N}}^T_k \otimes  N_k   \right )^{-1}   \left ( \check{\check{B}}^T \otimes B\right ) \ \myred{vec(I_{m})} $.
			\item[c.] $vec\left ( W \right ) = \left ( -\Lambda \otimes I_n - I_{r} \otimes A^T - \sum\limits_{k=1}^{m}\check{\check{N}}_k \otimes N^T_k \right )^{-1} \left ( \check{\check{C}}^T \otimes C^T\right ) \ \myred{vec(I_{p})}$.
			\item[d.] $V_r = orth\left ( V  \right ) , \ W_r = orth\left ( W \right ) $.
			\item[e.] $\check{A}= (W^T_r V_r)^{-1} W^T_r A V_r$, \quad  $\check{N}_k= \left ( W^T_r V_r \right )^{-1} W^T_r N_k V_r,$ 
			\item[] $ \check{B}=\left ( W^T_r V_r \right )^{-1} W^T_r B,$ \quad
			$ \check{C}= CV_r.$	
		\end{enumerate}
		\\ $A_r = \check{A}, \quad N_{k_r} = \check{N}_k, \quad B_r =\check{B}, \quad C_r = \check{C}$.
	\end{algorithmic}
\end{algorithm}
\section{Backward Stability}\label{Sec:BackwardStability}
\par 
In general, numerical algorithms for a problem are continuous in nature but, a digital computer solves them in a discrete manner. The reason is limitation on the representation of real / complex numbers. Since complex numbers can be represented by real numbers, we focus on latter only. Let $\textit{fd} :\mathbb{R}\rightarrow \mathbb{F}$ be a function giving a finite approximation to a real number. It provides rounded equivalent as \cite{Dirac8}
\begin{center}
	$\textit{fd} (x) = x (1+\epsilon_{machine})$   for all $x\in \mathbb{R}$,
\end{center} 
where $\epsilon_{machine}$ is the machine precision. Also, for every operation between any two finite numbers, the result is exact up to a relative error, i.e.,
for all x, y $ \in \mathbb{F} $
\begin{equation}
	\textit{fd} (x \oplus y) = (x \oplus y) (1+\epsilon_{machine}), \notag
\end{equation}
where $\oplus$ can be any of the following operation: $+,\ -,\ *,\textnormal{and}\ /$.
\par 
Consider a continuous mathematics algorithm $\textit{f}: X\rightarrow Y$. Say executing this algorithm on a digital computer (that uses finite precision arithmetic) is represented as $\widetilde{\textit{f}}: X\rightarrow Y$. To check how good the approximated algorithm $\widetilde{f}$ is, one usually computes the accuracy of $\widetilde{f}$. We say an algorithm $\widetilde{f}$ is accurate if \cite{Dirac8}
\begin{equation}
	  \frac{\left \| f(x) - \widetilde{f}(x)  \right \| }{\left \| f(x) \right \|} =\mathcal{O} (\epsilon_{machine}), \notag
\end{equation}
where $x \in X$. From the above equation, we find that computing accuracy is not possible since we do not know $f(x)$. A more easier parameter to check the goodness of $\widetilde{f}$ is stability. There are multiple notions of stability. One such notion is backward stability, which says that an algorithm $f$ is backward stable if \cite{Dirac8}
\begin{align*}
	\begin{split}
		& \widetilde{f}(x) =f( \widetilde{x}) \quad \textnormal{for some $\widetilde{x}$ with} \\ 
		& \qquad \frac{\|x - \widetilde{x} \|} {\|x\|} =  \mathcal{O}  ( \epsilon_{machine}). 
	\end{split}
\end{align*}
This notion of backward stability is useful since one can easily compute accuracy of the result/ output for a backward stable algorithm.
\begin{theorem}\cite{Dirac8}
	\label{theorem:BackwardStabilityDefinition}
	If $f:X \rightarrow Y$ is a backward stable algorithm, and $k(x)$ is the condition number of the problem, then the relative error 
	\begin{equation}
		\frac{\left \| f(x) -  \widetilde{f}(x) \right \|}{ \left \| f(x) \right \|} = \mathcal{O} \left ( \mathit{k}(x) \ \epsilon_{machine} \right ), \notag
	\end{equation}
	where $\epsilon_{machine}$ is the machine precision (or perturbation in $ x$).
\end{theorem}
\par 
Let's look at lines 3b. and 3c. in BIRKA (Algorithm \ref{BIRKAAlgo}). There we need to solve linear systems to compute \myworries{$vec(V)$ and $vec(W)$, respectively.} Solving these linear systems by direct methods (such as LU-factorization, Gaussian elimination, etc.) is too expensive (time complexity of $\mathcal{O} (n^3),$ where $n$ is the system size). Moreover the linear systems here have sparse matrices. 
For such systems, iterative methods, e.g., Krylov subspace methods \cite{Dirac6}, are preferred because of the reduced complexity (time complexity of $\mathcal{O}(n \times nnz),$ where $nnz$ is the number of nonzeros in the matrix) \footnote{In fact, the matrices here are block sparse. Iterative methods for difficult to solve linear systems usually require a preconditioner. Hence, this block sparsity can be exploited in designing preconditioners here. E.g., in \cite{luo2015optimization}, authors have designed an Incomplete LU (ILU) factorization for efficiently solving block sparse linear systems. The techniques from \cite{luo2015optimization}, can be used for designing better ILU preconditioned iterative methods for block sparse linear systems.}.
\par 
Iterative methods are inexact in nature, which means they do not solve linear systems, say {\fontfamily{qcr}\selectfont
Ax = b,}  exactly. Instead  {\fontfamily{qcr}\selectfont
Ax = b + $\delta$ } is solved, where $\delta$ is the stopping tolerance. Our aim is to find that if one uses an iterative solver (also called inexact solver from now on) in IRKA or BIRKA, are these algorithms stable with respect to \myworries{the error introduced by the inexact solves}. As earlier, we check for backward stability. For IRKA, the backward stability analysis has been done in \cite{Dirac2}.  
\par 
Let in BIRKA $vec(V)  \ \textnormal{and} \ vec(W)$ be calculated exactly, and $g$ be the functional representation of the interpolation process that uses $vec(V)$ and $vec(W)$ in BIRKA (i.e., exact BIRKA). Similarly, let $vec(\widetilde{V})$ and $ vec(\widetilde{W})$ be calculated inexactly (i.e., by an iterative solver), and $\widetilde{g}$ be the functional representation of the interpolation process that uses $vec(\widetilde{V})  \ \textnormal{and} \ vec(\widetilde{W})$ in BIRKA (i.e., inexact BIRKA). Then, from the backward stability definition, BIRKA is backward stable if
\begin{align}
	& \widetilde{g}(\zeta) = g(\widetilde{\zeta})   \quad \textnormal{for some $\widetilde{\zeta}$  with } \label{eq:14}  \\
	& \frac{\|\zeta- \widetilde{\zeta}\|_{H_2 \ or \ H_{\infty}}} {\|\zeta\|_{H_2 \ or \ H_{\infty}}} = \mathcal{O}(\|F\|), \label{eq:15} 
\end{align}
where $\widetilde{\zeta}$ is the perturbed full model corresponding to the error in the linear solves for  $vec(\widetilde{V})$ and $ vec(\widetilde{W})$ in inexact BIRKA. This perturbation is denoted by $F$. \myworries{Next, we look at the above two conditions for stability in the two different sub-sections below.}
\myworries{\subsection[Satisfying the First Condition of Backward Stability]{\textbf{Satisfying the First Condition of Backward Stability}} \label{Subsec:FirstCondition}}
\myworries{Let the original full order model be represented as $\zeta : A, \ N_1, \ \ldots,\ N_m, \ B, \ C$. Recall from Algorithm \ref{BIRKAAlgo}, the following:}
\begin{align} \label{eq:6.1}
\begin{split}
	vec\left ( V \right ) &= \left ( -\Lambda\otimes I_{n} -I_{r} \otimes A - \sum_{k=1}^{m} \check{\check{N_k}}^T \otimes  N_k \right )^{-1}  \left ( \check{\check{B}}^T \otimes B\right ) \ \myred{vec(I_{m})}  \quad \textnormal{and} \\
	vec\left ( W \right ) &= \left ( -\Lambda\otimes I_{n} -I_{r} \otimes A^{T} - \sum_{k=1}^{m}\check{\check{N}}_k\otimes  N^{T}_k \right )^{-1} \left ( \check{\check{C}}^T \otimes C^T\right )  \ \myred{vec(I_{p})}.
	\end{split}
\end{align}
\myworries{Also, let the residuals associated with iterative solves for computing $vec(\widetilde{V})$ and $vec(\widetilde{W})$ be $vec(R_{B})$ and  $vec(R_{C})$, respectively. Then, the above equations lead to}
\begin{align}
	\left ( -\Lambda\otimes I_{n} -I_{r} \otimes A-  \sum_{k=1}^{m} \check{\check{N_k}}^T \otimes  N_k \right )  vec\left ( \widetilde{V} \right )    = &  \left ( \check{\check{B}}^T \otimes B\right )  \ \myred{vec(I_{m})} + vec\left ( R_{B} \right ) \quad \textnormal{and} \label{eq:34} \\
	\left ( -\Lambda\otimes I_{n} -I_{r} \otimes A^{T} -  \sum_{k=1}^{m}\check{\check{N}}_k\otimes  N^{T}_k  \right )  vec\left ( \widetilde{W} \right )    = & \left ( \check{\check{C}}^T \otimes C^T\right )  \ \myred{vec(I_{p})}  + vec\left ( R_{C} \right ).\label{eq:35}
\end{align}
Let $\widetilde{V}_r = orth(\widetilde{V})$ and $\widetilde{W}_r = orth(\widetilde{W})$. The Petrov-Galerkin projection connects the reduced model matrices (obtained by inexact BIRKA) to the original full model matrices as  
\begin{align} \label{eq:32.1}
	  \begin{array}
      		[l]{l}
		\widetilde{A}_r = \left ( \widetilde{W}^{T}_r\widetilde{V}_r \right )^{-1}\widetilde{W}^{T}_rA\widetilde{V}_r,\ \
		\widetilde{N}_{k_r} = \left ( \widetilde{W}^{T}_r\widetilde{V}_r \right )^{-1}\widetilde{W}^{T}_r N_k \widetilde{V}_r,\ \\
		\widetilde{B}_r = \left ( \widetilde{W}^{T}_r\widetilde{V}_r \right )^{-1}\widetilde{W}^{T}_r B, \ \textnormal{and} \ \ 
		\widetilde{C}_r = C\widetilde{V}_r,
	  \end{array}
\end{align}
\myworries{where this reduced model is represented as $\widetilde{\zeta}_r : \widetilde{A}_r, \ \widetilde{N}_{1_r},\ \ldots, \ \widetilde{N}_{m_r}, \ \widetilde{B}_r,$ $ \widetilde{C}_r  $}.
\par 
By the backward stability definition, next we find a perturbed full model whose exact interpolation will give the reduced model as obtained by inexact interpolation of the original full model. \myred{Let the perturbed full model \myworries{be} represented as  $\widetilde{\zeta} : \widetilde{A}, \ \widetilde{N}_{1},\ \ldots, \ \widetilde{N}_{m}, \ \widetilde{B},\ \widetilde{C}  $  or $\widetilde{\zeta} : A+F, \ N_1 + E_1,\ \ldots, \ N_m + E_m, \ B+G,\ C+H $, where $F,\ E_1,\ \ldots,\ E_m,\ G,\ H$ are the constant perturbation matrices.} Then, we have
\begin{align}\label{allMatrixPerturbedEq}
	\left ( -\Lambda\otimes I_{n} -I_{r} \otimes\left ( A+F \right )-  \sum_{k=1}^{m} \check{\check{N_k}}^T \otimes  \left(N_k + E_k\right)  \right )  vec\left ( \widetilde{V} \right )   &=  \left ( \check{\check{B}}^T \otimes \left(B+G\right)\right ) \ \myred{vec(I_{m})} \quad \textnormal{and}   \\
	\left ( -\Lambda\otimes I_{n} -I_{r} \otimes \left ( A+F \right )^{T} -  \sum_{k=1}^{m}\check{\check{N}}_k\otimes  \left(N_k + E_k\right)^T  \right )  vec\left ( \widetilde{W} \right )  &= \left ( \check{\check{C}}^T \otimes \left(C + H\right)^T \right )  \ \myred{vec(I_{p})},\notag 
\end{align}
or
\begin{align}
\left ( -\Lambda\otimes I_{n} -I_{r} \otimes A-  \sum_{k=1}^{m} \check{\check{N_k}}^T \otimes  N_k \right )  vec\left ( \widetilde{V} \right ) = &  \left ( \check{\check{B}}^T \otimes B\right )  \ \myred{vec(I_{m})} +  \left ( \check{\check{B}}^T \otimes G\right )  \ \myred{vec(I_{m})} \notag \\ + & \left (  I_r \otimes F  + \sum_{k=1}^{m} \check{\check{N}}_k^T \otimes E_k\right) vec\left ( \widetilde{V} \right )    \quad \textnormal{and} \label{eq:38}\\
\left ( -\Lambda\otimes I_{n} -I_{r} \otimes A^{T} -  \sum_{k=1}^{m}\check{\check{N}}_k\otimes  N^{T}_k  \right )  vec\left ( \widetilde{W} \right )  &  =   \left ( \check{\check{C}}^T \otimes C^T\right )  \ \myred{vec(I_{p})} + \left ( \check{\check{C}}^T \otimes H^T\right )  \ \myred{vec(I_{p})} \notag \\ + &  \left (  I_r \otimes F^T  + \sum_{k=1}^{m} \check{\check{N}}_k \otimes E_k^T\right) vec\left ( \widetilde{W} \right ).\label{eq:39}
\end{align} 
\myworries{As earlier,  $\widetilde{V}_r = orth(\widetilde{V})$ and $\widetilde{W}_r = orth(\widetilde{W})$. Using the Petrov-Galerkin projection to connect the reduced model matrices (obtained by exact BIRKA) with the perturbed full model matrices we get}
\begin{align} \label{eq:40}
	\begin{array}
      	[l]{l}
	\widehat{A}_r = \left ( \widetilde{W}^{T}_r\widetilde{V}_r \right )^{-1}\widetilde{W}^{T}_r\left ( A+F \right )\widetilde{V}_r,  \ \ 
	\widehat{N}_{k_r}  = \left ( \widetilde{W}^{T}_r\widetilde{V}_r \right )^{-1}\widetilde{W}^{T}_r \left(N_k + E_k\right)  \widetilde{V}_r, \ \\
	\widehat{B}_r  = \left ( \widetilde{W}^{T}_r\widetilde{V}_r \right )^{-1}\widetilde{W}^{T}_r \left(B+G\right),  \ 
	\ \textnormal{and} \ \
	\widehat{C}_r = \left(C+H\right)\widetilde{V}_r,
	\end{array}
\end{align}
\myworries{where this reduced model is represented as $\widehat{\zeta}_r : \widehat{A}_r, \ \widehat{N}_{1_r},\ \ldots,\ \widehat{N}_{m_r},\ \widehat{B}_r,$\ $ \widehat{C}_r$. To satisfy the backward stability's first condition \eqref{eq:14}, we equate the reduced models in \eqref{eq:32.1} and \eqref{eq:40}. That is,
	\begin{align}
	\widehat{A}_r & = \left ( \widetilde{W}^{T}_r\widetilde{V}_r \right )^{-1}\widetilde{W}^{T}_r\left ( A+F \right )\widetilde{V}_r  = \left ( \widetilde{W}^{T}_r\widetilde{V}_r \right )^{-1}\widetilde{W}^{T}_r A \widetilde{V}_r + \left ( \widetilde{W}^{T}_r\widetilde{V}_r \right )^{-1}\widetilde{W}^{T}_r F \widetilde{V}_r  \notag \\
	& = \widetilde{A}_r +  \left ( \widetilde{W}^{T}_r\widetilde{V}_r \right )^{-1}\widetilde{W}^{T}_r F \widetilde{V}_r.  \notag
	\end{align}	
	Similarly, $\widehat{N}_{k_r}  = \widetilde{N}_{k_r} +  \left ( \widetilde{W}^{T}_r\widetilde{V}_r \right )^{-1}\widetilde{W}^{T}_r E_k \widetilde{V}_r$, \ $\widehat{B}_r  = \widetilde{B}_r +  \left ( \widetilde{W}^{T}_r\widetilde{V}_r \right )^{-1}\widetilde{W}^{T}_r G$ and $\widehat{C}_r  = \widetilde{C}_r + H\widetilde{V}_r$.}
	\par 
	\myworries{
	From the above, we note that if $\widetilde{W}^{T}_r F \widetilde{V}_r = 0$, then $\widehat{A}_r = \widetilde{A}_r$. Similarly, if $\widetilde{W}^{T}_r E_k \widetilde{V}_r = 0$, then $\widehat{N}_{k_r} = \widetilde{N}_{k_r}$;  if $\widetilde{W}^{T}_r G  = 0$, then  $\widehat{B}_{r} = \widetilde{B}_{r}$;  and if $H \widetilde{V}_r  = 0$, then $\widehat{C}_{r} = \widetilde{C}_{r}$. Using the Petrov-Galerkin framework for the inexact solves in \eqref{eq:34} and \eqref{eq:35}, we can easily achieve some of the above relations. We discuss this next.
}

\myworries{\subsubsection[The Petrov-Galerkin Framework for Inexact Solves Stability]{\textbf{The Petrov-Galerkin Framework for Inexact Solves}} \label{Subsubsec:PGFrameworkInexactSolves}}
\myworries{The Petrov-Galerkin framework by definition implies finding the solution of a linear system of equation such that its residual at every point is orthogonal to some other suitable subspace \cite{van2003iterative}. In our context, we define the Petrov-Galerkin framework as below.
	\begin{align}\label{eq:15.1}
	\begin{split}
	\textnormal{Find} &\quad \widetilde{V} \in \mathcal{P}_r \quad \textnormal{such that} \quad R_B \perp \mathcal{Q}_r \quad \textnormal{and} \\
	\textnormal{find} &\quad \widetilde{W} \in \mathcal{Q}_r \quad \textnormal{such that} \quad R_C \perp \mathcal{P}_r,
	\end{split}
	\end{align}	 
	where $\mathcal{P}_r$ and $\mathcal{Q}_r$ are any two r-dimensional subspaces of $\mathbb{C}^n$; $\widetilde{V}$ and $R_B$ satisfy \eqref{eq:34}; and $\widetilde{W}$ and $R_C$ satisfy \eqref{eq:35}.}
\par 
\myworries{Comparing \eqref{eq:34} with \eqref{eq:38} and \eqref{eq:35} with \eqref{eq:39}, we get the following equations:
	  	\begin{alignat}{2}
	  	vec\left ( R_{B} \right )& = \left ( \check{\check{B}}^T \otimes G\right )  \ \myred{vec(I_{m})}  +  \left (  I_r \otimes F  + \sum_{k=1}^{m} \check{\check{N}}_k^T \otimes E_k\right) vec\left ( \widetilde{V} \right )  \qquad \ \textnormal{and} \notag\\ \quad  vec\left (  R_{C}\right ) & = \left ( \check{\check{C}}^T \otimes H^T\right )  \ \myred{vec(I_{p})}  +   \left (  I_r \otimes F^T  + \sum_{k=1}^{m} \check{\check{N}}_k \otimes E_k^T\right) vec\left ( \widetilde{W} \right )  \quad \textnormal{or} \notag \\
	  	R_{B} & = G \check{\check{B}} +F\widetilde{V} + \sum_{k=1}^{m} E_k \widetilde{V} \check{\check{N}}_k  \qquad \textnormal{and}   \qquad  R_{C} = H^T \check{\check{C}}  + F^T \widetilde{W} + \sum_{k=1}^{m} E_k^T \widetilde{W} \check{\check{N}}_k^T . \label{eq:12.1}
	  	\end{alignat}}
	  \par 
	 \myworries{Next, we consider perturbations in $A, \ N_k, \ B$ and $C$ individually, and use the Petrov-Galerkin framework discussed above.
	 First, if we take the perturbation $F$ in $A$ only, then \eqref{eq:12.1} is equivalent to 
	\begin{alignat}{2}
	R_{B} & = F\widetilde{V}  \qquad \textnormal{and}   \qquad  R_{C}^T =   \widetilde{W}^T F . \label{eq:16}
	\end{alignat}
	In the above, if we multiply $\widetilde{W}^T$ from left in the first equation and $\widetilde{V}$ from right in the second equation, then we get
		\begin{alignat}{2}
	\widetilde{W}^T	R_{B}  = \widetilde{W}^T F\widetilde{V}  \qquad \textnormal{and}   \qquad  R_{C}^T \widetilde{V} =   \widetilde{W}^T F \widetilde{V}. \notag
		\end{alignat}
	From the Petrov-Galerkin framework \eqref{eq:15.1}, $\widetilde{W} \perp R_{B} \ \textnormal{and}	\ \widetilde{V} \perp  R_{C}$, and hence,
\begin{align} \label{eq:17}
		&\widetilde{W}^T F \widetilde{V} = 0   \qquad  \textnormal{or} \qquad \widetilde{W}^T_r F \widetilde{V}_r = 0. \footnotemark
\end{align}\footnotetext{Since $\widetilde{V}_r = orth(\widetilde{V})$  and $\widetilde{W}_r = orth(\widetilde{W})$, we have $\widetilde{V} = \widetilde{V}_r \ Z_1$ and $\widetilde{W} = \widetilde{W}_r \ Z_2$, where $Z_1$ and $Z_2$ are lower triangular matrices. Here $\widetilde{W}^T F \widetilde{V} = 0 $ implies $Z^T_2 \left( \widetilde{W}^T_r F \widetilde{V}_r \right) Z_1 = 0$. If $\widetilde{V}$ and $\widetilde{W}$ are full ranked then, $Z_1$ and $Z_2$ are invertible and we have $ \widetilde{W}^T_r F \widetilde{V}_r=0$. This full rank assumption exists in original BIRKA as well (see Lemma 5.2 in \cite{Dirac3}).}}
	\par 
	\myworries{Similarly, if we take the perturbation $E_k$ in {\it any one} $N_k$ matrix, then \eqref{eq:12.1} is equivalent to 
	\begin{alignat}{2}
	R_{B} =  E_k \widetilde{V} \check{\check{N}}_k  \qquad \textnormal{and}   \qquad  R_{C}^T = \check{\check{N}}_k \widetilde{W}^T   E_k. \notag
	\end{alignat}
	Again in the above, if we multiply $\widetilde{W}^T$ from left in the first equation and $\widetilde{V}$ from right in the second equation, then we get
	\begin{alignat}{2}
	\widetilde{W}^T	R_{B}  =  \widetilde{W}^T E_k \widetilde{V} \check{\check{N}}_k  \qquad \textnormal{and}   \qquad  R_{C}^T \widetilde{V} =   \check{\check{N}}_k \widetilde{W}^T   E_k \widetilde{V}. \notag
	\end{alignat}
	Using the Petrov-Galerkin framework \eqref{eq:15.1} in above we get 
	\begin{align*}
	\widetilde{W}^T E_k \widetilde{V} \check{\check{N}}_k =0 \qquad \textnormal{and}  \qquad
\check{\check{N}}_k \widetilde{W}^T   E_k \widetilde{V} = 0. 
	\end{align*}
	To achieve the desired result, i.e., $\widetilde{W}^{T}_r E_k \widetilde{V}_r = 0$, we need $\check{\check{N}}_k$ to be invertible. This cannot always be guaranteed. Thus, we drop the perturbation analysis with $N_k$ matrices.  
	}
	\par 
	\myworries{Finally, if we only take the perturbations $G$ and $H$, in the matrices $B$ and $C$, respectively, then \eqref{eq:12.1} is equivalent to 
		\begin{alignat}{2}
		R_{B} =  G \check{\check{B}} \qquad \textnormal{and}   \qquad  R_{C}^T =\check{\check{C}}^T H. \notag
		\end{alignat}
		As in the last two paragraphs, multiplying by $\widetilde{W}^T$ from left in the first equation above, multiplying by $\widetilde{V}$ from right in the second equation above, and using the Petrov-Galerkin framework \eqref{eq:15.1} we get   
		\begin{align*}
	\widetilde{W}^T  G \check{\check{B}}  =0 \qquad \textnormal{and}  \qquad
	\check{\check{C}}^T H \widetilde{V} = 0. 
		\end{align*}
		As above, to achieve the desired result, i.e., $\widetilde{W}^{T}_r G  = 0$ and $H \widetilde{V}_r  = 0$, we need $\check{\check{B}}$ and $\check{\check{C}}$ to be invertible. This cannot always be guaranteed because these are non-square matrices. Thus, we drop the perturbation analysis with $B$ and $C$ matrices both.  
		}
		\par
\myworries{Hence, \eqref{eq:17} implies that if we consider the perturbation in $A$ matrix only and use a Petrov-Galerkin framework for the inexact linear solves, then
	\begin{alignat}{2}
	\widehat{A}_r = \widetilde{A}_r, \quad \widehat{N}_{k_r} = \widetilde{N_{k_r}},& \quad \widehat{B}_r = \widetilde{B}_r, \quad \textnormal{and} \quad \widehat{C}_r= \widetilde{C}_r \quad \textnormal{or} \notag \\
	\widetilde{g}(\zeta) &= g(\widetilde{\zeta}). \notag
	\end{alignat}
		The theorem below summarizes this.} \\
		\myworries{
	\begin{theorem}
		\label{theorem:FirstCondition}	
		If the inexact linear solves in BIRKA (line 3b. and 3c. of Algorithm \ref{BIRKAAlgo}) are solved using the Petrov-Galerkin framework \eqref{eq:15.1}, then BIRKA satisfies the first condition of backward stability with respect to these solves, i.e., \eqref{eq:14}. 
	\end{theorem}
}
\myworries{\subsection[Satisfying the Second Condition of Backward Stability]{\textbf{Satisfying the Second Condition of Backward Stability}}\label{Subsec:SecondCondition}} 
\par
\myworries{Next, we show that the \textit{second condition} of backward stability, given in \eqref{eq:15}, is also satisfied. According to \eqref{eq:15}, the difference between the original full model and the perturbed full model should be order of the perturbation, i.e.,  }
\begin{align}
	\qquad \frac{\|\zeta- \widetilde{\zeta}\|_{H_2 \ or \ H_{\infty}}} {\|\zeta\|_{H_2 \ or \ H_{\infty}}} & = \mathcal{O}(\|F\|). \notag  
\end{align}
We satisfy the above condition in the absolute sense, since $\zeta$ is independent of $F$. That is,
	\begin{alignat}{2}
	\left \| \zeta -\widetilde{\zeta} \right \|_{H_2}^{2}  =  \mathcal{O}\Big(\|F\|_2 \Big). \notag
	\end{alignat} 
\par 
	Consider the error system $\zeta^{err} = \zeta - \widetilde{\zeta}$ whose matrices are defined as follows \cite{Dirac3,Dirac5}: 
	\begin{center}
		$ A^{err} = \begin{bmatrix}
		A & 0\\ 
		0 & A+F
		\end{bmatrix},\quad
		N_k^{err} = \begin{bmatrix}
		N_k & 0\\ 
		0 & N_k
		\end{bmatrix},\quad
		B^{err} = \begin{bmatrix}
		B\\ 
		B
		\end{bmatrix},\quad
		\textnormal{and } \quad
		C^{err} = \begin{bmatrix}
		C & -C 
		\end{bmatrix} $.
	\end{center} 
	The $H_2-$norm of this error system is  
	\begin{align}\label{normErrorSystemComplexEq}
	\left \| \zeta^{err} \right \|_{H_2}^{2}  &=  vec(I_{2p})^T
	\Bigg( \begin{bmatrix}
	C & -C
	\end{bmatrix} \otimes \begin{bmatrix}
	C & -C
	\end{bmatrix} \Bigg) \  \times \notag\\ & \left ( -  \begin{bmatrix}
	A & 0\\ 
	0 & A+F
	\end{bmatrix}   \otimes    \begin{bmatrix}
	I_n & 0\\ 
	0 & I_n
	\end{bmatrix}     -  \begin{bmatrix}
	I_n & 0\\ 
	0 & I_n
	\end{bmatrix}   \otimes    \begin{bmatrix}
	A & 0\\ 
	0 & A+F
	\end{bmatrix}    - \sum_{k=1}^{m} \begin{bmatrix}
	N_k & 0\\ 
	0 & N_k
	\end{bmatrix}   \otimes    \begin{bmatrix}
	N_k & 0\\ 
	0 & N_k
	\end{bmatrix}        \right )^{-1} \notag\\
	& \times   
	\Bigg( \begin{bmatrix}
	B\\ 
	B
	\end{bmatrix} \otimes \begin{bmatrix}
	B\\ 
	B
	\end{bmatrix}  \Bigg)
	vec(I_{2m}), 
	 \end{align}
	 
	\begin{align}	 
	=\ & vec(I_{2p})^T
	\Bigg( \begin{bmatrix}
	C & -C
	\end{bmatrix} \otimes \begin{bmatrix}
	C & -C
	\end{bmatrix} \Bigg) \ \times \notag\\ &  \left( -  \begin{bmatrix}
	A & 0\\ 
	0 & A
	\end{bmatrix}   \otimes    \begin{bmatrix}
	I_n & 0\\ 
	0 & I_n
	\end{bmatrix}      -  \begin{bmatrix}
	I_n & 0\\ 
	0 & I_n
	\end{bmatrix}   \otimes    \begin{bmatrix}
	A & 0\\ 
	0 & A
	\end{bmatrix}     - \sum_{k=1}^{m} \begin{bmatrix}
	N_k & 0\\ 
	0 & N_k
	\end{bmatrix}   \otimes    \begin{bmatrix}
	N_k & 0\\ 
	0 & N_k
	\end{bmatrix}   - 
	\right.\notag\\
	& \left.
	\begin{bmatrix}
	0 & 0\\ 
	0 & F
	\end{bmatrix}   \otimes    \begin{bmatrix}
	I_n & 0\\ 
	0 & I_n
	\end{bmatrix}     -  \begin{bmatrix}
	I_n & 0\\ 
	0 & I_n
	\end{bmatrix}   \otimes    \begin{bmatrix}
	0 & 0\\ 
	0 & F
	\end{bmatrix}   \right)^{-1}  \times   
	\Bigg( \begin{bmatrix}
	B\\ 
	B
	\end{bmatrix} \otimes \begin{bmatrix}
	B\\ 
	B
	\end{bmatrix}  \Bigg)
	vec(I_{2m}). \notag
	\end{align}
	Let 
	\begin{alignat}{2}
	\widehat{C} &=\Bigg ( \begin{bmatrix}
	C & -C
	\end{bmatrix} \otimes \begin{bmatrix}
	C & -C
	\end{bmatrix} \Bigg ),  \notag\\
\widehat{Q} &=	\Biggl( -  \begin{bmatrix}
	A & 0\\ 
	0 & A
	\end{bmatrix}   \otimes    \begin{bmatrix}
	I_n & 0\\ 
	0 & I_n
	\end{bmatrix}     -  \begin{bmatrix}
	I_n & 0\\ 
	0 & I_n
	\end{bmatrix}   \otimes    \begin{bmatrix}
	A & 0\\ 
	0 & A
	\end{bmatrix}    - \mathlarger{\sum}_{k=1}^{m} \begin{bmatrix}
	N_k & 0\\ 
	0 & N_k
	\end{bmatrix}   \otimes    \begin{bmatrix}
	N_k & 0\\ 
	0 & N_k
	\end{bmatrix}        \Biggr), \label{eq:19.1}\\
	\widehat{F} &= \begin{bmatrix}
	0 & 0\\ 
	0 & F
	\end{bmatrix}, \label{eq:21.1}\\
	\widehat{\widehat{F}} &=\left( I_{2n} \otimes \widehat{F} +  \widehat{F} \otimes I_{2n}\right), \ \textnormal{and} \label{eq:20} \\
 \widehat{B} &=	\Bigg( \begin{bmatrix}
	B\\ 
	B
	\end{bmatrix} \otimes \begin{bmatrix}
	B\\ 
	B
	\end{bmatrix}  \Bigg). \notag
	\end{alignat}
	Then, the norm of this error system is
	\begin{align}\label{normErrorSystemSimpleEq}
	\left \| \zeta^{err} \right \|_{H_2}^{2} & = vec(I_{2p})^T
	 \widehat{C}  \left ( \widehat{Q} - \widehat{\widehat{F}} \right)^{-1}  \widehat{B} \ vec(I_{2m}), \\
	& = vec(I_{2p})^T
	 \widehat{C}  \widehat{Q}^{-1}  \left ( I_{4n^2} - \widehat{\widehat{F}} \widehat{Q}^{-1}   \right)^{-1}  \widehat{B} \  vec(I_{2m}), \notag
	\end{align}
	If $	\left \| \widehat{\widehat{F}}  \widehat{Q}^{-1} \right \|_2  <  1$, then by the Neumann series we get that	
	\begin{align}
	\left \| \zeta^{err} \right \|_{H_2}^{2} 
	& = vec(I_{2p})^T
	 \widehat{C}  \widehat{Q}^{-1}  \left ( I_{4n^2} -\widehat{\widehat{F}} \widehat{Q}^{-1}   \right)^{-1}  \widehat{B} \ vec(I_{2m}), \notag \\
	& = vec(I_{2p})^T
	 \widehat{C}  \widehat{Q}^{-1}  \left ( I_{4n^2} +\widehat{\widehat{F}} \widehat{Q}^{-1} + \left(\widehat{\widehat{F}} \widehat{Q}^{-1} \right)^2 +\ldots      \right)  \widehat{B} \ vec(I_{2m}), \notag\\
	& = vec(I_{2p})^T  \widehat{C}  \widehat{Q}^{-1}  \widehat{B} \ vec(I_{2m})  \notag \\
	& \quad +  vec(I_{2p})^T  \widehat{C}  \widehat{Q}^{-1}  \widehat{\widehat{F}} \widehat{Q}^{-1} 
	\left ( I_{4n^2} +\widehat{\widehat{F}} \widehat{Q}^{-1} + \left(\widehat{\widehat{F}} \widehat{Q}^{-1} \right)^2 + \ldots      \right)  \widehat{B} \ vec(I_{2m}). \notag 
	\end{align}
	Since $\left \| \zeta - \zeta \right \|_{H_2}^2= vec(I_{2p})^T  \widehat{C}  \widehat{Q}^{-1}  \widehat{B} \ vec(I_{2m}) = 0$, the above equation simplifies to
	\begin{align}
	\left \| \zeta^{err} \right \|_{H_2}^{2} = vec(I_{2p})^T  \widehat{C}  \widehat{Q}^{-1}  \widehat{\widehat{F}} \widehat{Q}^{-1} 
	\left ( I_{4n^2} +\widehat{\widehat{F}} \widehat{Q}^{-1} + \left(\widehat{\widehat{F}} \widehat{Q}^{-1} \right)^2 +\ldots     \right)  \widehat{B} \ vec(I_{2m}). \label{eq:28}
	\end{align}
Bounding the right hand side of the above equation we get the following:
	\begin{align}
	&\left | vec(I_{2p})^T  \widehat{C}  \widehat{Q}^{-1}  \widehat{\widehat{F}} \widehat{Q}^{-1}  \left ( I_{4n^2} -\widehat{\widehat{F}} \widehat{Q}^{-1}   \right)^{-1}  \widehat{B}  \ vec(I_{2m})  \right |   \notag\\ 
	&  \qquad  \leq \left \| vec(I_{2p})^T \right \|  \left \| \widehat{C}  \widehat{Q}^{-1} \right \|  \left \|  \widehat{\widehat{F}} \right \|  \left \| \widehat{Q}^{-1}  \right \|    
	\left \|   \left ( I_{4n^2} -\widehat{\widehat{F}} \widehat{Q}^{-1}   \right)^{-1} \right \|    \left \|  \widehat{B}  \right \| \left \| vec(I_{2m}) \right \|, \notag \\
	&  \qquad  \leq \left \| vec(I_{2p})^T \right \|  \left \| \widehat{C} \widehat{Q}^{-1} \right \|  \left \|  \widehat{\widehat{F}} \right \|  \left \| \widehat{Q}^{-1}  \right \|    
	\left ( \frac{1}{1-\left \| \widehat{\widehat{F}}  \widehat{Q}^{-1} \right \|}  \right )  \left \|  \widehat{B}  \right \| \left \| vec(I_{2m}) \right \|  \ldots (\textnormal{Using lemma 2.3.3 in \cite{golub2012matrix}}). \notag 
	\end{align}
	Substituting the above in \eqref{eq:28} we get
	\begin{alignat}{2}
\left\| \zeta^{err}\right\|_{H_2}^2	 & \leq \left \| vec(I_{2p})^T \right \|  \left \| \widehat{C}  \widehat{Q}^{-1} \right \|  \left \|  \widehat{\widehat{F}} \right \|  \left \| \widehat{Q}^{-1}  \right \|    
	\left ( \frac{1}{1-\left \| \widehat{\widehat{F}}  \widehat{Q}^{-1} \right \|}  \right )  \left \|  \widehat{B}  \right \|  \left \| vec(I_{2m}) \right \|. \label{eq:18.1}
	\end{alignat}
	\myred{
Let $\left\|\widehat{Q}^{-1} \right\| < 1$, which is defined by the original system (further analyzed in Section \ref{Subec:Invetible}) and $\left\|\widehat{\widehat{F}} \right\| < 1$, which is related to the residuals of linear solves (further analyzed in Section \ref{Subsec:Accuracy}). Then, using the matrix norm property we have the following:	
} 
\myred{
	\begin{alignat}{2}
	\left \| \widehat{\widehat{F}}  \widehat{Q}^{-1} \right \| & \leq \left \| \widehat{\widehat{F}} \right \|  \left \| \widehat{Q}^{-1} \right \|  \ \textnormal{or} \notag \\
	 \frac{1}{1 - \left \| \widehat{\widehat{F}}  \widehat{Q}^{-1} \right \| } & \leq \frac{1}{1 - \left \| \widehat{\widehat{F}} \right \|  \left \| \widehat{Q}^{-1} \right \| }. \notag
	\end{alignat}
	Substituting the above in \eqref{eq:18.1} we get
	\begin{alignat}{2}
	\left\| \zeta^{err}\right\|_{H_2}^2	 & \leq \left \| vec(I_{2p})^T \right \|  \left \| \widehat{C}  \widehat{Q}^{-1} \right \|  \left \|  \widehat{\widehat{F}} \right \|  \left \| \widehat{Q}^{-1}  \right \|    
	\left (  \frac{1}{1 - \left \| \widehat{\widehat{F}} \right \|  \left \| \widehat{Q}^{-1} \right \| } \right )   \left \|  \widehat{B}  \right \|  \left \| vec(I_{2m}) \right \| \label{eq:21}
	\end{alignat}	
	or
	\begin{alignat}{2}
	\left\| \zeta^{err}\right\|_{H_2}^2 \leq \mathcal{O} \left( \left\| \widehat{\widehat{F}} \right\| \right). \label{eq:17.1}
	\end{alignat}
	Next, we relate $\left\| \widehat{\widehat{F}} \right\|$ and $\left\| F \right\|$. From \eqref{eq:20} we know 
	 \begin{alignat}{2}
	 \widehat{\widehat{F}} =\left( I_{2n}\otimes \widehat{F} + \widehat{F} \otimes I_{2n} \right). \notag
	 \end{alignat}
	 Taking norms on both the sides of the above equation, and  applying the triangle inequality property ( $\left \| X+Y \right \| \leq \left \| X \right \| + \left \| Y \right \|$) we get 
	 \begin{alignat}{3}
	 \left \| \widehat{\widehat{F}} \right \|  &= \left \| I_{2n} \otimes \widehat{F} + \widehat{F} \otimes  I_{2n} \right \| 
	 \leq \left \| I_{2n} \otimes \widehat{F}  \right \| + \left \| \widehat{F} \otimes  I_{2n} \right \|.  \notag 
	 \end{alignat}
	 Further, using the norm distribution property of Kronecker product  ( $\left \| X \otimes Y  \right \| =  \left \| X  \right \|  \left \|  Y  \right \|$) \cite{matrixanalysis, normkron}, we have the following:
	 \begin{alignat}{3}
	 \left \| \widehat{\widehat{F}} \right \|  & \leq \left \| I_{2n}  \right \|  \left \| \widehat{F} \right \|   + \left \| \widehat{F} \right \|  \left \| I_{2n}  \right \|,   \notag\\
	 &  \leq \mathcal{O} \left( \left \| \widehat{F} \right \| \right).		 \notag
	 \end{alignat}
	 From \eqref{eq:21.1} we know $ \widehat{F} = \begin{bmatrix}
	 0 & 0\\ 
	 0 & F
	 \end{bmatrix}$. 
	 Using the definitions of all the commonly used matrix norms (Frobenius, 2, 1 and $\infty$) \cite{meyerBook} we get 
	 \begin{alignat}{2}
	 \mathcal{O}\Big(\|\widehat{\widehat{F}}\|\Big) \leq \mathcal{O}\Big(\|F\| \Big). \label{eq:27}
	 \end{alignat}		
	 Substituting the above in \eqref{eq:17.1} we get
	 \begin{alignat}{2}
	 \left \| \zeta^{err} \right \|_{H_2}^{2}  = \left \| \zeta -\widetilde{\zeta} \right \|_{H_2}^{2} \leq \mathcal{O}\Big(\|F\| \Big). \notag
	 \end{alignat}}
	Thus, we have satisfied the second condition of backward stability. The  theorem below summarizes this.
\myworries{\begin{theorem}
		\label{theorem:SecondCondition}
		If $\widehat{Q}$ defined in \eqref{eq:19.1} is invertible, $\left\|\widehat{Q}^{-1} \right\|  <  1$, and $ \left\| \widehat{\widehat{F}}\right\| < 1$, where $\widehat{\widehat{F}}$ is defined in \eqref{eq:20}, then BIRKA satisfies the second condition of backward stability with respect to the inexact linear solves, i.e.,  \eqref{eq:15}. 
	\end{theorem}	
  The hypotheses of this theorem are usually easy to satisfy, and are discussed in the next section. The corollary below summarizes our stability result.}
	\myworries{
		\begin{corollary}\label{corollary:BackwardStability}
		Assuming the hypotheses of Theorem \ref{theorem:FirstCondition} and Theorem \ref{theorem:SecondCondition} are satisfied, then BIRKA is  backward stable with respect to the inexact linear solves.
		\end{corollary}}
	\myworries{\section[Analysis]{Analysis} \label{Sec:Analysis}
	Next, we analyze our assumptions and results from the previous sections. First, we revisit the assumed invertibility of all relevant matrices (in Section \ref{Subec:Invetible}). Second, we derive the expression for accuracy of the reduced system, in-terms of the residuals of the linear solves as well as the conditioning of the bilinear system (in Section \ref{Subsec:Accuracy}).}
\myworries{	  
		\subsection[Invertibility of Involved Matrices]{\textbf{Invertibility of Involved Matrices}}\label{Subec:Invetible}
	Until now, we have assumed invertibility of eight matrices. Most of these invertibility assumptions directly come from the control system theory as well as the model reduction theory of bilinear systems. We have also assumed invertibility of few newly proposed matrices. In this subsection, we summarize/ analyze all these assumptions in the order of appearance of the corresponding matrix in this paper. We first summarize the invertibility assumptions from literature.  		   
		\begin{enumerate}[label=(\alph*)]
			\item We assume invertibility of $\left(s_k I_n- A\right)$ and $\left(sI_n - A\right)$ in \eqref{transferfunBilinearEq} and \eqref{transferfunLinearEq}, respectively. These come from the transfer function definitions. Please see Section 2 of \cite{doi:10.1137/130947830} and Section 1 of \cite{Dirac4}, respectively.			
			\item In the ${H}_2-$norm definition of a bilinear dynamical system \eqref{h2normDefEq}, we assume that $\big ( - A \otimes I_n - I_n \otimes A - \sum_{k=1}^{m} N_k \otimes N_k \big)$ is invertible. This is a standard definition. Please see Theorem 3.4  of \cite{Dirac3}.			
			\item We assume invertibility of $\left( \widetilde{W}^T_r \widetilde{V}_r\right)$ in \eqref{redSystemEq}. As mentioned earlier, this is easy to enforce and come from BIRKA. Please see Algorithm 2 of \cite{Dirac3} or Algorithm 1 of \cite{doi:10.1137/130947830}.			
			\item In \eqref{eq:4},  we assume the middle term, i.e., \\ 
			$\left ( -\begin{bmatrix}A & 0\\ 0 & {\Lambda}\end{bmatrix} \otimes \begin{bmatrix}I_n & 0 \\ 0 & I_r\end{bmatrix} - \begin{bmatrix}I_n & 0\\ 0 & I_r\end{bmatrix} \otimes \begin{bmatrix}A & 0\\ 0 & \check{A}\end{bmatrix} -  \sum_{k=1}^{m}\begin{bmatrix}N_k & 0\\ 0 & \check{\check{N}}^T_k \end{bmatrix} \otimes \begin{bmatrix} N_k & 0\\ 0 & \check{N}_k\end{bmatrix}   \right )$ 
			is invertible. This comes from the $H_2-$norm of the error system ($\zeta -\zeta_r$). Please see Corollary 4.1 of \cite{Dirac3} or Theorem 4.5 of \cite{Dirac5}.
			\item We assume invertibility of $\big ( -\Lambda\otimes I_{n} -I_{r} \otimes  A -\sum_{k=1}^{m} \check{\check{N_k}}^T \otimes N_k  \big )$ in Algorithm 1. This again comes from BIRKA. Please see Algorithm 2 of \cite{Dirac3} or Algorithm 1 of \cite{doi:10.1137/130947830}.
		\end{enumerate}
	   During the backward stability analysis of BIRKA, we assume invertibility of some newly proposed matrices. Next, we analyze these matrices. Note that below, we discuss the matrix in (b) before the matrix in (c) although the latter appears first in this paper. This is done for ease of exposition. 
		\begin{enumerate}[label=(\alph*)] 
			\item In IRKA \cite{Dirac4}, $\left(sI - A \right)$ is inverted to form the projection subspace. Hence, in the backward stability analysis of IRKA, invertibility of the corresponding perturbed matrix $\left(sI - \left(A+F\right) \right)$ is assumed (see Theorem 4.1 of \cite{Dirac2}). As discussed in (e) above, in BIRKA, $\big ( -\Lambda\otimes I_{n} -I_{r} \otimes  A -\sum_{k=1}^{m} \check{\check{N_k}}^T \otimes N_k  \big )$ is inverted to form the projection subspace. Hence, we assume invertibility of the corresponding perturbed matrix $\big ( -\Lambda\otimes I_{n} -I_{r} \otimes\left ( A+F \right )-\sum_{k=1}^{m} \check{\check{N_k}}^T \otimes (N_k + E_k)   \big )$ in \eqref{allMatrixPerturbedEq}.
			\item We assume invertibility of $\widehat{Q}$ given in (\ref{eq:19.1}). Also listed below for easy access. 
			\begin{alignat}{2}
		  \widehat{Q} =	-\begin{bmatrix} A & 0\\ 0 & A \end{bmatrix} \otimes \begin{bmatrix} I_n & 0\\ 0 & I_n \end{bmatrix} - \begin{bmatrix} I_n & 0\\ 0 & I_n \end{bmatrix} \otimes \begin{bmatrix} A & 0\\ 0 & A \end{bmatrix} - \sum_{k=1}^{m} \begin{bmatrix} N_k & 0\\ 0 & N_k \end{bmatrix} \otimes \begin{bmatrix} N_k & 0\\ 0 & N_k \end{bmatrix}. \notag
			\end{alignat}
			This is one of the most important assumption in obtaining a backward stable BIRKA (see Corollary \ref{corollary:BackwardStability}). Hence, here we relate this invertibility assumption with the underlying bilinear dynamical system.
			If we define $A_2 = \begin{bmatrix} A & 0\\ 0 & A \end{bmatrix}$,  $I_{2n} = \begin{bmatrix} I_n & 0\\ 0 & I_n \end{bmatrix}$, $N_{2k} = \begin{bmatrix} N_k & 0\\ 0 & N_k \end{bmatrix}$ and $\widehat{Q} = Q_1 \otimes Q_2$, where $Q_1,Q_2 \in \mathbb{R}^{2n \times 2n}$ are any two matrices, then $\widehat{Q}$ can be rewritten as
			\begin{alignat}{4}
			-A_2 \otimes I_{2n} - I_{2n} \otimes A_2 - \sum_{k=1}^{m} N_{2k} \otimes N_{2k} &= Q_1 \otimes Q_2  \qquad \qquad \textnormal{or}  \notag \\
			-\left(  	A_2 \otimes I_{2n}  \right)  vec(I_{2n})                   -\left(  I_{2n} \otimes A_2 \right)  vec(I_{2n})
			- \sum_{k=1}^{m} \left( N_{2k} \otimes N_{2k} \right)  vec(I_{2n})
			&= \left( Q_1 \otimes Q_2 \right)  vec(I_{2n})\ \ \textnormal{or}   \notag 	\\
			-A_2^T -A_2 - \sum_{k=1}^{m} N_{2k} N_{2k}^T &= Q_2 Q_1^T  \qquad \textnormal{or} \notag\\
			-  \begin{bmatrix} A & 0\\ 0 & A \end{bmatrix}^T -  \begin{bmatrix} A & 0\\ 0 & A \end{bmatrix}  - \sum_{k=1}^{m} \begin{bmatrix} N_k & 0\\ 0 & N_k \end{bmatrix} \begin{bmatrix} N_k & 0\\ 0 & N_k \end{bmatrix}^T &= Q_2 Q_1^T \qquad \textnormal{or}  \notag \\
				\begin{bmatrix}
				-A^T -A - \sum_{k=1}^{m} N_k N_k^T & 0 \\ 
				0 & -A^T -A - \sum_{k=1}^{m} N_k N_k^T
				\end{bmatrix} &= Q_2 Q_1^T. \notag
			\end{alignat}	
			If $\left( -A^T -A - \sum_{k=1}^{m} N_k N_k^T \right)$ is invertible, then $Q_1$ and $Q_2$ are invertible. This implies that $\widehat{Q} = \left(Q_1 \otimes Q_2\right)$ is invertible. Consider the following generalized Lyapunov equation used in the derivation of BIRKA \cite{Dirac3, heattransfermodel}:
			\begin{align}
				AP + PA^T + \sum_{k=1}^{m} N_k P N_k^T = -BB^T. \notag
			\end{align}		
			 If the solution of this equation is the identity matrix (i.e., $P=I_n$), then the left hand side matrix in this Lyapunov equation is $ \left( A^T + A + \sum_{k=1}^{m} N_k N_k^T \right)$, which needs to be invertible for invertibility of $\widehat{Q}$. 				
			\item In \eqref{normErrorSystemComplexEq} and \eqref{normErrorSystemSimpleEq}, we assume invertibility of   
			\begin{alignat}{2}
				\left( -\begin{bmatrix} A & 0\\ 0 & A+F \end{bmatrix} \otimes \begin{bmatrix} I_n & 0\\ 0 & I_n \end{bmatrix} - \begin{bmatrix} I_n & 0\\ 0 & I_n \end{bmatrix} \otimes \begin{bmatrix} A & 0\\ 0 & A+F \end{bmatrix} - \sum_{k=1}^{m} \begin{bmatrix} N_k & 0\\ 0 & N_k \end{bmatrix} \otimes \begin{bmatrix} N_k & 0\\ 0 & N_k \end{bmatrix} \right) \notag
			\end{alignat}
			 and $\left(\widehat{Q} - \widehat{\widehat{F}}\right)$, respectively, both of which represent the same matrix (i.e., $\widehat{Q}$ with perturbation). This matrix is invertible if $  \Big( -\left(A+F\right)^T -\left(A+F \right) - \sum_{k=1}^{m} N_kN^T_k \Big) $ is invertible. 
		\end{enumerate}
}
\myworries{
	\subsection[Accuracy of the Reduced System]{\textbf{Accuracy of the Reduced System}} \label{Subsec:Accuracy} 
	Assume that BIRKA satisfies the hypotheses of Corollary \ref{corollary:BackwardStability}, i.e., it is backward stable with respect to the inexact linear solves. Then, from Theorem \ref{theorem:BackwardStabilityDefinition} we get that 
	\begin{align*}
		\dfrac{\left \|g \left(\zeta \right) - \widetilde{g} \left(\zeta \right) \right \|_{H_2} }{\left \| g \left(\zeta \right) \right \|_{H_2}} = \mathcal{O} \left ( \mathit{k}(\zeta) \ \|F\| \right ), 
	\end{align*}
	where, as earlier (recall \eqref{eq:14}-\eqref{eq:15}), $g$ denotes exact BIRKA, $\widetilde{g}$ denotes inexact BIRKA, $\zeta$ is the original full model, $k(\zeta)$ is the condition number of $\zeta$ (discussed below), and $F$ is the perturbation in $\zeta$.
}
\par 
\myworries{
	 If we define, $g\left( \zeta \right) =\zeta_r $, and  $\widetilde{g} \left( \zeta \right)=\widetilde{\zeta}_r $, then the above equation can be rewritten as
	\begin{align*}
		\frac{\left \|\zeta_r - \widetilde{\zeta}_r\right \|_{H_2} }{\left \| \zeta_r \right \|_{H_2}} = \mathcal{O} \left ( \mathit{k}(\zeta) \ \|F\| \right ).  
	\end{align*}
	Here, we are looking at the reduced systems obtained at line 3e. of Algorithm \ref{BIRKAAlgo}, i.e., at the end of every iterative step of BIRKA. Thus, accuracy of the reduced system is dependent on the conditioning of the problem as well as the perturbation. Next, we look at both these quantities separately. 
}
\par 
\myworries{ 
	\textit{First}, we want to compute conditioning of our bilinear system with respect to performing the inexact linear solves on lines 3b. and 3c. of Algorithm \ref{BIRKAAlgo}. 
Since for backward stability we equate the reduced model obtained by performing inexact BIRKA on the original full model (${\zeta}$) and performing exact BIRKA on the perturbed full model ($\widetilde{\zeta}$), these inexact linear solves are captured by $\widetilde{\zeta}$. Thus, the condition number of our bilinear system with respect to computing the $H_2-$norm of the error system $\zeta_{err} = \zeta - \widetilde{\zeta}$ will give us a \textit{good approximation} to the condition number that we want to compute (with respect to computing the $H_2-$norm of  $\widetilde{\zeta}_r - \zeta$ or $\widetilde{\zeta}_r - \zeta_r$). Similar behavior has been observed for linear dynamical systems (see Theorem 3.1 and 3.3 in \cite{Dirac2})}.
\par 
\myworries{
	Recall, the condition number by definition means relative change in the output (for us this is $\dfrac{\left \| \zeta - \widetilde{\zeta} \right \|_{H_{2}}}{\left \| \zeta \right \|_{H_{2}}}$) with respect to the relative change in the input (for us this is $\dfrac{\left \| F \right \|}{\left \| A \right \|}$ since we are perturbing the $A$ matrix). Hence, from \eqref{eq:21} we have 
	\begin{alignat}{2}
	\left \| \zeta - \widetilde{\zeta} \right \|_{H_2} \leq \left \| vec\left ( I_{2p} \right )^T \right \|        \left \| \widehat{C} \widehat{Q}^{-1} \right \|  \left \| \widehat{Q}^{-1} \right \|   \left \| \widehat{B} \right \|  \left \| vec(I_{2m}) \right \| 
	\frac{\left \| \widehat{\widehat{F}} \right \|}{1 - \left \| \widehat{\widehat{F}} \right \|  \left \| \widehat{Q}^{-1} \right \| }, \label{eq:22}
	\end{alignat}	 
	where $\left \| \widehat{Q}^{-1} \right \| < 1$ and $\left \| \widehat{\widehat{F}}\right \| < 1 $. Since $\left\| \widehat{\widehat{F}}\right\| < 1$, then we also have
	\begin{alignat}{2}
	\frac{1}{1 - \left \| \widehat{\widehat{F}} \right \|  \left \| \widehat{Q}^{-1} \right \|} \leq \frac{1}{1- \left\|\widehat{Q}^{-1} \right\|}.\notag
	\end{alignat}
Using above, \eqref{eq:22} can be rewritten as
	\begin{alignat}{2}
	\left \| \zeta - \widetilde{\zeta} \right \|_{H_2} &\leq \left \| vec\left ( I_{2p} \right )^T \right \|        \left \| \widehat{C} \widehat{Q}^{-1} \right \|  \left \| \widehat{Q}^{-1} \right \|   \left \| \widehat{B} \right \|  \left \| vec(I_{2m}) \right \| \frac{\left \| \widehat{\widehat{F}} \right \|}{1- \left\|\widehat{Q}^{-1} \right\|} {\text\quad or}\notag \\
	\frac{\left \| \zeta - \widetilde{\zeta} \right \|_{H_2} }{ \left\| \zeta \right\|_{H_2}} &\leq \frac{\left \| vec\left ( I_{2p} \right )^T \right \|        \left \| \widehat{C} \widehat{Q}^{-1} \right \|  \left \| \widehat{Q}^{-1} \right \|   \left \| \widehat{B} \right \|  \left \| vec(I_{2m}) \right \| \left\| A \right\| }{\left\| \zeta \right\|_{H_2}} \frac{1}{1- \left\|\widehat{Q}^{-1} \right\|}  \frac{\left\| \widehat{\widehat{F}} \right\|}{\left\| A \right\|}. \notag 
	\end{alignat}
	From \eqref{eq:27}, we know $\|\widehat{\widehat{F}}\| \leq \|F\|$. Hence, the above inequality is equivalent to
	\begin{alignat}{2}
		\frac{\left \| \zeta - \widetilde{\zeta} \right \|_{H_2} }{ \left\| \zeta \right\|_{H_2}} &\leq k \left(\zeta  \right) \frac{\left\| F \right\|}{\left\| A \right\|}, \notag 
	\end{alignat}
	where
	\begin{alignat}{2}
	k \left(\zeta  \right) =   \frac{\left \| vec\left ( I_{2p} \right )^T \right \|        \left \| \widehat{C} \widehat{Q}^{-1} \right \|  \left \| \widehat{Q}^{-1} \right \|   \left \| \widehat{B} \right \|  \left \| vec(I_{2m}) \right \| \left\| A \right\| }{\left\| \zeta \right\|_{H_2}} \frac{1}{1- \left\|\widehat{Q}^{-1} \right\|}. \label{eq:25}
	\end{alignat}
	In the numerical experiments section, for both our problems, we show that this condition number is fairly small. In other words, both our problems are well-conditioned with respect to computing the $H_2-$norm of the error system $\zeta_{err}$. Note that $\left\| \widehat{Q}^{-1}\right\| < 1$ and $\left\|\widehat{\widehat{F}} \right\| < 1$ as assumed here come from the assumptions for backward stability of BIRKA (see Corollary \ref{corollary:BackwardStability}), and hence, we do not need any extra assumptions.
}
\par
\myworries{
	\textit{Second}, we relate the perturbation $F$ with the residuals $R_B$ and $R_C$ given in \eqref{eq:34} and \eqref{eq:35}, respectively. Recall that we are considering the perturbation $F$ in $A$ matrix, and hence, this $F$ should satisfy both the equations in \eqref{eq:16}. That is, \begin{alignat}{2}
	R_{B} & = F\widetilde{V}  \qquad \textnormal{and}   \qquad  R_{C}^T =   \widetilde{W}^T F . \label{eq:30}
	\end{alignat} }
\myworries{
	From the assumptions for backward stability of BIRKA (Corollary \ref{corollary:BackwardStability}), we know that we need to use a Petrov-Galerkin framework, i.e., 
	\begin{align}
	 \widetilde{W} \perp R_{B} \qquad \textnormal{and}	\qquad \widetilde{V} \perp  R_{C}, \label{eq:31}
	\end{align}	
	where $\widetilde{V}$ and $\widetilde{W}$ are again given in \eqref{eq:34} and \eqref{eq:35}, respectively. Using \eqref{eq:31}, we get that
	\begin{align}
		F =  R_{B} \left ( \widetilde{W}^{T} \widetilde{V} \right )^{-1} \widetilde{W}^{T} +  \widetilde{V} \left ( \widetilde{W}^{T} \widetilde{V} \right )^{-1}R_{C}^{T}, \label{eq:32}
	\end{align}
	satisfies \eqref{eq:30}. This is assuming $\left( \widetilde{W}^T \widetilde{V}\right)$ is nonsingular, which has already been discussed in the previous subsection. The theorem below gives a bound on this $F$. This theorem is similar to Theorem 4.2 from \cite{Dirac2} in the linear case.
	}
	\myworries{
	\begin{theorem}
		\label{theorem6}
		Let $R_B$ and $\widetilde{V}$ be defined as in \eqref{eq:34}, $R_C$ and $\widetilde{W}$ be defined as in \eqref{eq:35}, and $F$ be defined as in \eqref{eq:32}.	Define 	${R}_B = [R_{B_1},\ R_{B_2},\ \ldots,\ R_{B_r}]$ and \	${R}_C = [R_{C_1}, \ R_{C_2},\ \ldots, \ R_{C_r}]$ 	
		and assume $\widetilde{W}^T \widetilde{V}$ is nonsingular. Then, the perturbation $F$ satisfies
		\begin{align*}
		\begin{array}
		{l}
		\left \| F \right \|_2 \leq\left \| F \right \|_F 
		\leq \sqrt{r} \left \{ \underset{i}{max} \left \| R_{B_i} \right \| \left \| (\widetilde{W}^T\widetilde{V})^{-1} \widetilde{W}^T \right \| +  \underset{i}{max} \left \| R_{C_i} \right \|  \left \| \widetilde{V}(\widetilde{W}^T\widetilde{V})^{-1} \right \| \right \}.
		\end{array}
		\end{align*}		
	\end{theorem}
	\begin{proof}
		Note that
		\begin{align}
		F &=  R_{B} \left ( \widetilde{W}^{T} \widetilde{V} \right )^{-1} \widetilde{W}^{T} +  \widetilde{V} \left ( \widetilde{W}^{T} \widetilde{V} \right )^{-1}R_{C}^{T}.\notag\\
		\left \| F \right \|_F &=  \left \| R_{B} \left ( \widetilde{W}^{T} \widetilde{V} \right )^{-1} \widetilde{W}^{T} +  \widetilde{V} \left ( \widetilde{W}^{T} \widetilde{V} \right )^{-1} R_{C}^{T} \right \|_F \notag \\
		\left \| F \right \|_F & \leq \left \| R_{B} \left ( \widetilde{W}^{T} \widetilde{V} \right )^{-1} \widetilde{W}^{T} \right \|_F + \left \| \widetilde{V} \left ( \widetilde{W}^{T} \widetilde{V} \right )^{-1}R_{C}^{T}  \right \|_F. \notag 
		\end{align}
		Consider the first term from the above expression as
		\begin{align}
		\left \| R_{B}\left ( \widetilde{W}^{T} \widetilde{V} \right )^{-1} \widetilde{W}^{T} \right \|_F & \leq \bigl \|  R_{B}  \bigr \|_F  \left \| \left ( \widetilde{W}^{T} \widetilde{V} \right )^{-1} \widetilde{W}^{T} \right \| \notag \\
		& \leq \sqrt{r}\ \underset{i}{max} \bigl \|  R_{B_i}  \bigr \|  \left \| \left ( \widetilde{W}^{T} \widetilde{V} \right )^{-1} \widetilde{W}^{T} \right \|. \notag    
		\end{align}
		Similarly, taking the second term as
		\begin{align}
		\left \|  \widetilde{V} \left ( \widetilde{W}^{T} \widetilde{V} \right )^{-1} R_{C}^T \right \|_F  & \leq \left \| \widetilde{V} \left ( \widetilde{W}^{T} \widetilde{V} \right )^{-1} \right \|  \bigl \|  R_{C}  \bigr \|_F  \notag \\
		& \leq \sqrt{r}\ \underset{i}{max} \bigl \|  R_{C_i}  \bigr \|  \left \| \widetilde{V} \left ( \widetilde{W}^{T} \widetilde{V} \right )^{-1} \right \|. \notag
		\end{align}
		Finally, we get
		\begin{align*}
		\begin{array}
		{l}
		\left \| F \right \|_2 \leq\left \| F \right \|_F 
		\leq \sqrt{r} \left \{ \underset{i}{max} \left \| R_{B_i} \right \|  \left \| (\widetilde{W}^T \widetilde{V})^{-1} \widetilde{W}^T \right \| +  \underset{i}{max} \left \| R_{C_i} \right \|  \left \| \widetilde{V}(\widetilde{W}^T\widetilde{V})^{-1} \right \| \right \}.
		\end{array}
		\end{align*}
	\end{proof}
}
\par 
In the expression of $\left \| F \right \|$ above, we see that the norm of the perturbation is proportional to the norm of the two residuals obtained while solving the two set of linear systems ($\left \| R_B \right \|$ and $ \left \| R_C \right \|$) as well as the norm of two other quantities  $\Big(\left \| (\widetilde{W}^T \widetilde{V})^{-1} \widetilde{W}^T \right \|$ and $\left \| \widetilde{V}(\widetilde{W}^T\widetilde{V})^{-1} \right \| \Big)$.
 These two quantities are very less dependent on accuracy of the linear systems we solve. They are also not sensitive to different initializations of BIRKA as well as different reduced system sizes. This behavior is similar to the related quantities obtained in the stability analysis of IRKA \cite{Dirac2}. We support this argument with numerical experiments in Section \ref{Subsec:HeatTransfer}.
\par 
\myworries{ 
To summarize, $\left \|\zeta_r - \widetilde{\zeta}_r\right \|_{H_2} $ is proportional to $k\left(\zeta\right)$ and $\left \| F \right \|$. The problem is usually well conditioned, and $\left \| F \right \|$ is directly proportional to $\left \| R_B \right \|$ and $ \left \| R_C \right \|$. Thus, as we iteratively solve the linear systems arising in BIRKA more accurately (i.e., reduce the stopping tolerance of the linear solver), we get a more accurate reduced system. This is very useful in deciding on when to stop the linear solver. If we need a very accurate reduced system, then we need to iterate more in the linear solver, else we can stop earlier. We support this with numerical experiments in the next section.
}
\section{Numerical Experiments} \label{Sec:Experiment}
\par 
We perform experiments to support the conjecture, as discussed above, on two models. First, we use a flow model \cite{breiten2010krylov}
in Section \ref{Subsec:FlowModel}, and then we use a heat transfer model \cite{Dirac3,heattransfermodel} in Section \ref{Subsec:HeatTransfer}. These models give us both SISO as well as MIMO bilinear dynamical systems of sizes varying from $100$ to $40,000$. 
\par 
The resulting linear systems to be solved vary from $600 \times 600$ to $2,00,000 \times 2,00,000$. For solving the linear systems while computing $V$ and $W$ by a direct method (exact BIRKA), we use a backslash in Matlab. This uses Gaussian elimination as the underlying algorithm. The most popular iterative methods for solving the sparse linear systems of equations are the Krylov subspace methods \cite{saad2003iterative}. \myworries{As discussed in Section \ref{Subsec:FirstCondition}, for a backward stable BIRKA with respect to the inexact linear solves, we need to use a linear solver based upon the Petrov-Galerkin framework (Theorem \ref{theorem:FirstCondition} and Corollary \ref{corollary:BackwardStability}).} Since \myworries{the Biconjugate Gradient (BiCG)} algorithm \cite{ahuja2012recycling} is an iterative linear solver based upon this framework, we use it for solving the linear systems while computing $V$ and $W$ by an iterative method (inexact BIRKA), i.e., $\widetilde{V}$ and $\widetilde{W}$.
\par 
We implement our codes in MATLAB (2015a),  and test on a machine with the following configuration: Intel Xeon(R) CPU E5-1620 V3 @ 3.50 GHz., frequency 1200 MHz., 8 CPU, 64 GB RAM.
\subsection{\textbf{A Flow Model}} \label{Subsec:FlowModel}
We first do experiments on a ``flow model" \cite{breiten2010krylov}, which consists of a one dimensional viscid Burgers equation. That is,
\begin{align}
	\frac{\partial w}{\partial t} + w \frac{\partial w}{\partial x} &= \frac{\partial }{\partial x}  \left (v \frac{\partial w}{\partial x} \right ),  \qquad \textnormal{for}\ 
	\left ( x,\ t \right ) \in \left ( 0,\ L \right ) \times \left ( 0,\ T \right ), \notag \\
	w\left ( 0,\ t \right ) &= u\left ( t \right ), \qquad \textnormal{for} \ t \in \left ( 0,\ T \right ), \notag
\end{align}
where $w(x,\ t)$ is the velocity at a particular point $x$ and a time $t$;
and $v(x,\ t)$ is the viscosity coefficient that we take as a constant $(v)$. We perform spatial \myred{semi-}discretization of the \myworries{above} equation with equidistant step size $h= \dfrac{L}{N+1}$, where N is the number of interior points in the interval (0,\ L). \myworries{Further, using Carleman bilinearization \cite{Dirac3,breiten2010krylov}, we obtain a bilinear 
dynamical system of order $N \times N^2$. We briefly show these steps below.} 
\begin{equation}
	\frac{\mathrm{d} }{\mathrm{d} t}  \begin{bmatrix}
	w_1\\ 
	w_2\\ 
	\cdot\\ 
	\cdot\\ 
	w_i\\
	\cdot\\
	\cdot\\ 
	w_N  	
	\end{bmatrix}
	= \begin{bmatrix}
	\dfrac{-w_1w_2 }{2h} + \dfrac{v}{h^2}(w_2 -2 w_1)  \\ 
	\dfrac{-w_2 }{2h} (w_3-w_1) + \dfrac{v}{h^2}(w_3-2w_2 + w_1) \\ 
	\cdot\\
	\cdot\\  
	\dfrac{-w_i }{2h} (w_{i+1}-w_{i-1}) + \dfrac{v}{h^2}(w_{i+1}-2w_{i} + w_{i-1}) \\  
	\cdot\\ 
	\cdot\\ 
	\dfrac{-w_Nw_{N-1} }{2h} + \dfrac{v}{h^2}(-2w_N + w_{N-1})
	\end{bmatrix}
	+ \begin{bmatrix}
	\dfrac{w_1}{2h} + \dfrac{v}{h^2}\\ 
	0\\ 
	\cdot\\ 
	\cdot\\ 
	0\\ 
	\cdot\\ 
	\cdot\\ 
	0
	\end{bmatrix} u \notag
\end{equation} 
or
\begin{align}
	\frac{\mathrm{d} w}{\mathrm{d} t} =  f(w) + g(w) u, \notag
\end{align}
where $\omega = \begin{bmatrix}
\omega_1, \ \omega_2, \ldots, \ \omega_N
\end{bmatrix}^T$; and $f(w)$ and $g(w)$ can be written in Kronecker product form as below.
\begin{align}
	f(w) & = A_1 w + \frac{1}{2} A_2 (w\otimes w), \notag \\
	g(w) & = B_0 +B_1 w, \notag
\end{align}     
where $B_0 \in \mathbb{R}^{N\times 1}$; $A_1, B_1 \in \mathbb{R}^{N\times N}$ are the Jacobians of $f(w) \ \textnormal{and} \ g(w)$, respectively; and $A_2 \in \mathbb{R}^{N\times N^2}$ is the second derivative of $f(w)$.
Let
\begin{equation*}
		\dot{x} = \frac{\mathrm{d} x}{\mathrm{d} t}  \qquad \textnormal{and} \qquad \dot{\omega} = \frac{\mathrm{d} \omega}{\mathrm{d} t}.
\end{equation*} 
Finally, we get the bilinear system of order $N+N^2$ as
\begin{align*}
	\dot{x} &= 
	\begin{bmatrix}
		A_1 &\frac{1}{2}A_2 \\ 
		0 & A_1 \otimes I + I \otimes A_1 
	\end{bmatrix}
	x
	+ \begin{bmatrix}
		B_1 & 0\\ 
		B_0 \otimes I + I \otimes B_0& 0 
	\end{bmatrix} 
	xu +
	\begin{bmatrix}
		B_0\\ 
		0
	\end{bmatrix}u, \\	
	y & = \frac{1}{N} \begin{bmatrix}
		\underbrace{1 \cdots  1}_\text{N times} & \underbrace{0 \cdots \cdots 0}_{N^{2}\ \text{times}}
	\end{bmatrix} x,
\end{align*}
where 
\begin{center}
	$	x= \begin{bmatrix}
	w\\
	w \otimes w
	\end{bmatrix}  \qquad \textnormal{and} \qquad \dot{x}= 
	\begin{bmatrix}
	\dot{w}\\
	\dot{w} \otimes w + w \otimes \dot{w} 
	\end{bmatrix}$.
\end{center}
We refer the reader to \cite{breiten2010krylov} for exact structure of $A_1,\ A_2,\ B_0$ and $B_1$.
\par 
 For our experiments, we take N = 10, L = 1 and $v=0.1$ that gives us a SISO bilinear dynamical system of size $110$. \myworries{We initialize the input system in BIRKA by random matrices} based upon similar setup in \cite{Dirac3} and \cite{Dirac5}. The stopping tolerance for BIRKA is taken as $10^{-6}$, and we reduce this model to size $6$. Both of these are again chosen based upon similar values in \cite{Dirac3} and \cite{Dirac5}. This leads to solving the linear systems of size $660 \times 660$. While using BiCG we use two different stopping tolerances ($10^{-2}$ and $10^{-8}$). Ideally, we should obtain a more accurate reduced model when using the smaller BiCG tolerance. 
 \par 
 \myworries{ 
 	First, let us look at the remaining assumptions for backward stability of BIRKA (see Theorem \ref{theorem:SecondCondition} and Corollary \ref{corollary:BackwardStability}). $\widehat{Q}$ is invertible here. We also have $ \left\| \widehat{Q}^{-1}\right \|$ less than one (i.e., $1.6051 \times 10^{-3}$). Finally, $\left\| \widehat{\widehat{F}}\right\|$, at the end of the first BIRKA step, for the BiCG stopping tolerance of $10^{-2}$ and $10^{-8}$ is $3.0675 \times 10^{-1}$ and $2.4596 \times 10^{-4}$, respectively, both of which are also less than one. These values are less than one at the end of all the other BIRKA steps as well.  The condition number for our problem, as defined in  \eqref{eq:25}, is $1.2125 \times 10^{-2}$. This shows that the flow model is well-conditioned.}
 \par 
\myworries{The accuracy results are given in Figure \ref{fig1} and Table \ref{table1}. In Figure \ref{fig1}, we have accuracy of the reduced system $\left(\left \| \zeta_r - \widetilde{\zeta}_r \right \|_{H_2}\right)$ on the y-axis and the BIRKA iterations on the x-axis. Table \ref{table1} gives the corresponding data.
From Figure \ref{fig1}, we do not observe any difference in the values of $\left(\left \| \zeta_r - \widetilde{\zeta}_r \right \|_{H_2}\right)$ for the two BiCG tolerances. The dotted line, which corresponds to the BiCG stopping tolerance $10^{-2}$  and the solid line, which corresponds to the BiCG stopping tolerance $10^{-8}$ coincide. Ideally, the solid line should be below the dotted line. This behaviour is clearly reflected in Table \ref{table1} (see the second and the fourth columns).}
\par 
\myworries{
	BIRKA gets more consistent as it converges to the ideal interpolation points. Hence, towards the end of the BIRKA iterations (iteration 14 to iteration 20), accuracy of the reduced system for the BiCG stopping tolerance of $10^{-8}$ is substantially better than accuracy of the reduced system for the BiCG stopping tolerance of $10^{-2}$. This is reflected in Figure \ref{fig2}, which is an enlarged version of Figure \ref{fig1}, as well as Table \ref{table1}. }
\begin{figure*}[h]
	\centering
	\captionsetup{font=small}
	\includegraphics[width=1\textwidth]{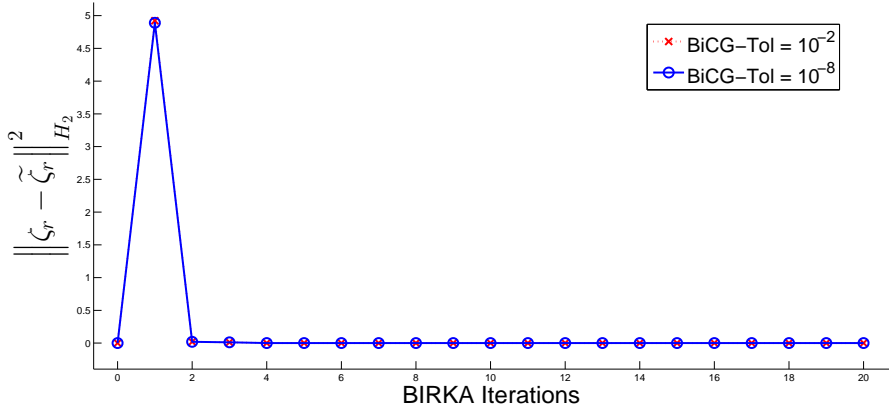}
	\caption{Accuracy of the reduced system plotted at each BIRKA iteration for the two different stopping tolerances in BiCG; flow model of size 110.}
	\label{fig1}     
\end{figure*}
\par 
In Table \ref{table1}, we observe that BiCG takes exactly same number of iterative steps from the BIRKA iteration \myred{$8$} until convergence. That is, for the BiCG stopping tolerance of $10^{-2}$ it stays at $44$, and for the BiCG stopping tolerance of $10^{-8}$ it stays at $90$. The reason for this is that the linear systems change very little from the \myred{$8^{th}$} BIRKA step. This can be inferred by looking at the eigenvalue distribution of the linear system matrices as well as their Frobenius norm.
\begin{figure*}[!]
	\centering
	\captionsetup{font=small}
	\includegraphics[width=1\textwidth]{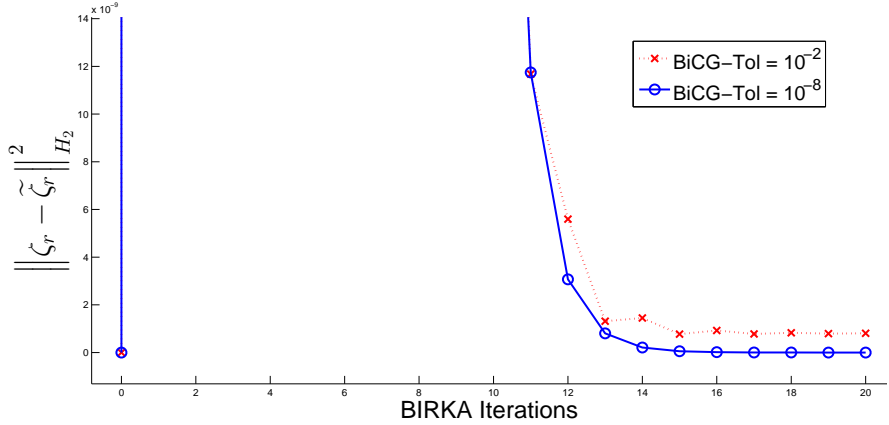}
	\caption{Enlarged Figure 1 towards the end of BIRKA iterations.}
	\label{fig2} 
\end{figure*}
\par 
\myworries{Figure \ref{fig3} shows the distribution of the six smallest eigenvalues (in absolute sense) of the linear system matrices corresponding to the BiCG stopping tolerance of $10^{-2}$ at the BIRKA steps $8,\ 9 \ \textnormal{and} \ 10$. Each of these six eigenvalues do not seem to change with respect to the change in the BIRKA steps. However,  if we look at any one eigenvalue, specifically, for example the smallest eigenvalue at the three different BIRKA steps, then we observe that it does change, but only slightly (see Figure \ref{fig4}). The Frobenius norm of the linear system matrices at the BIRKA steps \myworries{ $8,\ 9 \ \textnormal{and} \ 10$ are $1.7263 \times 10^{3}, \ 1.7264 \times 10^{3} \ \textnormal{and} \ 1.7266 \times 10^{3}, $} respectively. Thus, this  supports the argument that matrices do not change much.}
  
\begin{table}[h]
	\centering
	\captionsetup{font=small}
	\begin{tabular}{|c|c|c|c|c|}
		\hline
		\multirow{2}{*}{\begin{tabular}[c]{@{}c@{}}BIRKA \\ Iteration\end{tabular}} & \multicolumn{2}{c|}{BiCG-Tol of $10^{-2}$}                                                                          & \multicolumn{2}{c|}{BiCG-Tol of $10^{-8}$}                                                                          \\ \cline{2-5} 
		& $\left \| \zeta_r - \widetilde{\zeta}_r \right \|_{H_2}^2$ & \begin{tabular}[c]{@{}c@{}}BiCG \\ Iteration \\ Count\end{tabular} & $\left \| \zeta_r - \widetilde{\zeta}_r \right \|_{H_2}^2$ & \begin{tabular}[c]{@{}c@{}}BiCG\\  Iteration\\  Count\end{tabular} \\ \hline 	
	1  & $4.9214$                 & 91 & $4.8904$                 & 167 \\ \hline
	2  & $1.9671 \times 10^{-2}$  & 35 & $1.9649 \times 10^{-2}$  & 85  \\ \hline
	3  & $1.1745 \times 10^{-2}$  & 40 & $1.1735 \times 10^{-2}$  & 85  \\ \hline
	4  & $2.0764 \times 10^{-4}$  & 41 & $2.0583 \times 10^{-4}$  & 92  \\ \hline
	5  & $4.3239 \times 10^{-5}$  & 42 & $4.2785 \times 10^{-5}$  & 89  \\ \hline
	6  & $1.0181 \times 10^{-5}$  & 39 & $9.8618 \times 10^{-6}$  & 89  \\ \hline
	7  & $2.6412 \times 10^{-6}$  & 39 & $2.5583 \times 10^{-6}$  & 82  \\ \hline
	8  & $6.9999 \times 10^{-7}$  & 44 & $6.5685 \times 10^{-7}$  & 90  \\ \hline
	9  & $1.7325 \times 10^{-7}$  & 44 & $1.7213 \times 10^{-7}$  & 90  \\ \hline
	10 & $5.3043 \times 10^{-8}$  & 44 & $4.4857 \times 10^{-8}$  & 90  \\ \hline
	11 & $1.1675 \times 10^{-8}$  & 44 & $1.1745 \times 10^{-8}$  & 90  \\ \hline
	12 & $5.5945 \times 10^{-9}$  & 44 & $3.0702 \times 10^{-9}$  & 90  \\ \hline
	13 & $1.3127 \times 10^{-9}$  & 44 & $8.0359 \times 10^{-10}$ & 90  \\ \hline
	14 & $1.4474 \times 10^{-9}$  & 44 & $2.1026 \times 10^{-10}$ & 90  \\ \hline
	15 & $7.7234 \times 10^{-10}$ & 44 & $5.5041 \times 10^{-11}$ & 90  \\ \hline
	16 & $9.2674 \times 10^{-10}$ & 44 & $1.4398 \times 10^{-11}$ & 90  \\ \hline
	17 & $7.8030 \times 10^{-10}$ & 44 & $3.7841 \times 10^{-12}$ & 90  \\ \hline
	18 & $8.2925 \times 10^{-10}$ & 44 & $9.8779 \times 10^{-13}$ & 90  \\ \hline
	19 & $7.9294 \times 10^{-10}$ & 44 & $2.5543 \times 10^{-13}$ & 90  \\ \hline
	20 & $8.0646 \times 10^{-10}$ & 44 & $6.6835 \times 10^{-14}$ & 90             \\ \hline
	\end{tabular}
	\caption{Accuracy of the reduced system and BiCG iterations at each BIRKA step for the two different stopping tolerances in BiCG; flow model of size 110.}
	\label{table1}
\end{table}

\begin{figure}[h]
	\centering
	\captionsetup{font=small}
	\includegraphics[width=1\textwidth]{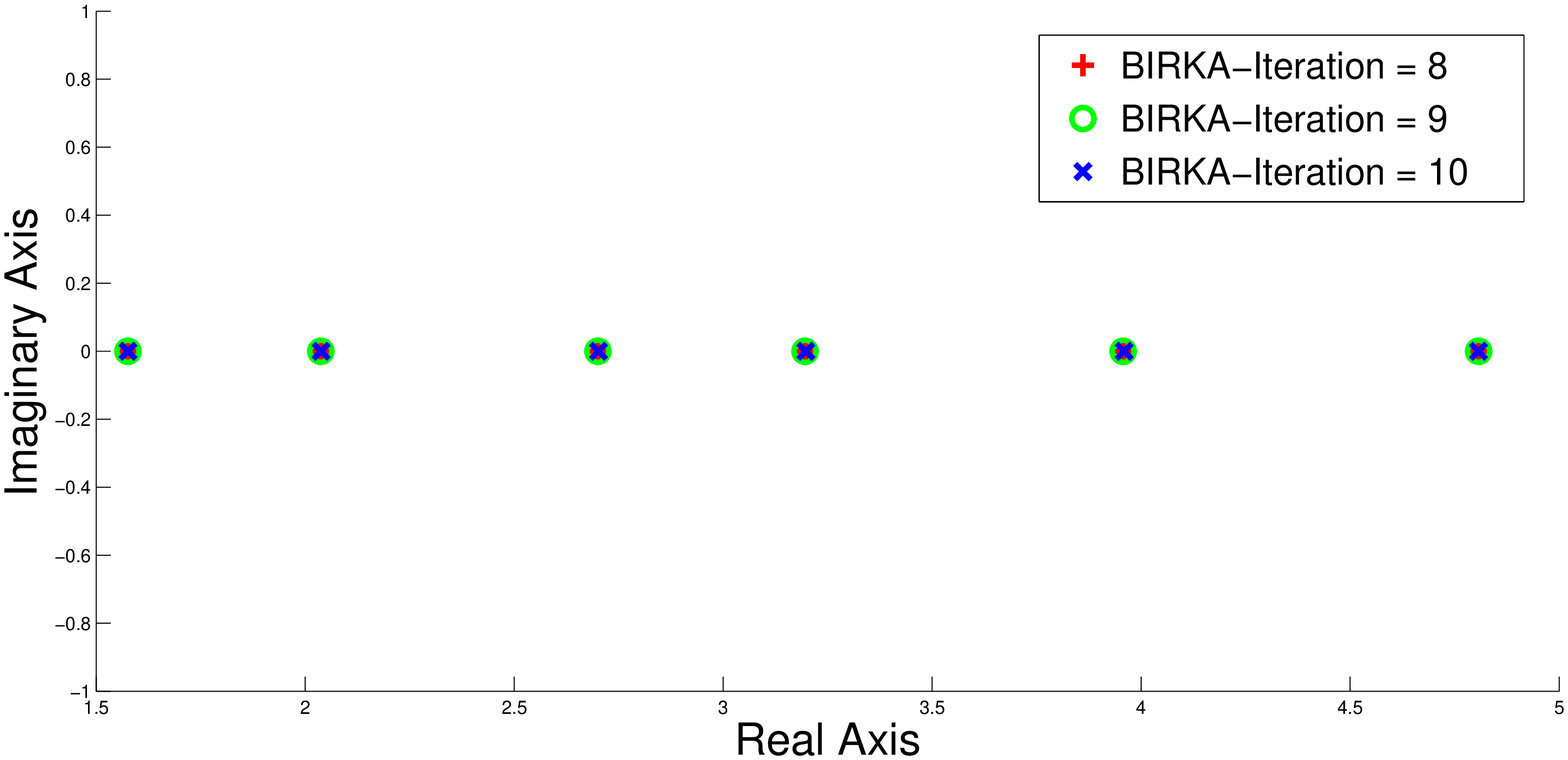} 
	\caption{The six smallest eigenvalues of the linear systems at the different BIRKA iterations.}
	\label{fig3}
\end{figure}
\begin{figure}[!]
	\centering
	\captionsetup{font=small}
	\includegraphics[width=1\textwidth]{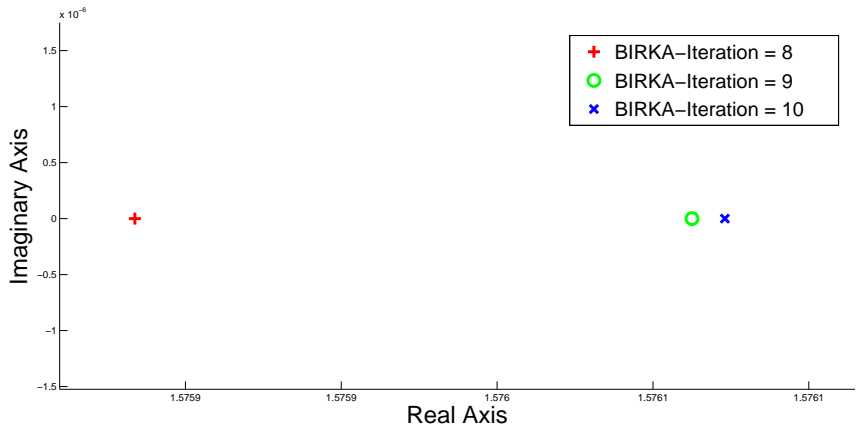} 
	\caption{Enlarged Figure 3 for the smallest eigenvalue.}
	\label{fig4}
\end{figure}

\subsection{\textbf{A Heat Transfer Model}} \label{Subsec:HeatTransfer}
	The next set of experiments we do on a heat transfer model as given below \cite{Dirac3,heattransfermodel}.
	 \begin{align}
	 x_t = \Delta x  &   \quad \textnormal{in} \ \left[ 0,1 \right ] \times \left [ 0,1 \right ], \notag\\
	 n.\nabla x = u_1 \left ( x-1 \right )  & \quad \textnormal{on} \ \Gamma_1
	  := \left \{ 0 \right \}\times (0,1)
	 ,\notag\\
	 n.\nabla x = u_2 \left ( x-1 \right )  & \quad \textnormal{on} \ \Gamma_2
	  :=  (0,1) \times \left \{ 0 \right \}
	 , \notag\\
	 x=0  &  \quad \textnormal{on} \ \Gamma_3
	 :=\left \{ 1 \right \} \times \left [ 0,1 \right ] 
	 \ \textnormal{and} \ \Gamma_4 
	 := \left[ 0,1 \right] \times \left \{ 1 \right \}
	 , \notag
	 \end{align}
	 where $x(l_1,\ l_2,\ t)$ is the temperature at a particular point in the space $(l_1,\ l_2)$ and at a time $t$; $n$ is the unit outward normal to the domain; $u_1$ and $u_2$ are the input variables; and $\Gamma_1,\ \Gamma_2, \ \Gamma_3,$ and $\Gamma_4$ are the boundaries of the unit square. After spatial discretization of 
	 \myred{the above equation} using $K^2$ grid points, we obtain a bilinear dynamical system of order $K^2 \times K^2$ with two inputs and one output as shown below.
	 \begin{align}
	 \dot{x} &= Ax +u_1 N_1 x + u_2 N_2 x + B u, \notag \\
	 y &= Cx, \notag
	 \end{align}
	 where, as earlier,
	 \begin{align}
	 \dot{x} &= \frac{\mathrm{d} x}{\mathrm{d} t},\quad u= \begin{bmatrix}
	 u_1\\ 
	 u_2
	 \end{bmatrix}, \notag\\
	 A &= \dfrac{1}{h^2} \left ( I_K \otimes T_K +T_K \otimes I_K +E_1 \otimes I_K + I_K \otimes E_K\right ),\notag\\ 
	 N_1 & = \dfrac{1}{h}\left ( E_1 \otimes I_K \right ),\ N_2 = \dfrac{1}{h}\left ( I_K \otimes E_K \right ),\notag\\
	 B &=\left [ \dfrac{1}{h}\left ( e_1 \otimes e \right ) \qquad \dfrac{1}{h}\left ( e \otimes e_K \right )\right], \ \textnormal{and} \ C = \dfrac{1}{K^2}\left ( e \otimes e \right )^T  \notag
	 \end{align}
	 with $I_K$ being the identity matrix of size $K$,
\begin{center}
		$ T_K =\begin{bmatrix}
	 -2 & 1 &  &  &  &\\ 
	 1& -2 & 1 &  & &\\ 
	 & \cdot & \cdot  & \cdot & & \\ 
	 &  & \cdot & \cdot & \cdot & \\ 
	 & & & 1 & -2 & 1 \\
	 &  &  &  & 1 & -2
	 \end{bmatrix} \in \mathbb{R}^{K\times K}, \qquad $
\end{center}
	 $E_j = e_j e_j^T$, the grid size $h= \dfrac{1}{K+1}$, $e_j$ is the $j^{th}$ column of the identity matrix $I_K$, and $e= [1,\ \ldots,\ 1] \in \mathbb{R}^K$.
\par 
We perform experiments on the heat transfer model for three different sizes, i.e., n = 100, 10,000 and 40,000 corresponding to K = 10, 100 and 200, respectively. \myworries{We initialize the input system in BIRKA by random matrices based upon the similar setup in \cite{Dirac3} and \cite{Dirac5}. The stopping tolerance for BIRKA is taken as $10^{-3}$. The size to which we reduce is different for the different model sizes, and is discussed below. Both these settings (the BIRKA stopping tolerance and the size of reduced system) are chosen based upon similar values in \cite{Dirac3,Dirac5}. While using BiCG (unpreconditioned for smaller size and preconditioned for larger sizes), we use two different stopping tolerances ($10^{-4}$ and $10^{-8}$). Ideally, as discussed earlier, we should obtain a more accurate reduced model for the smaller stopping tolerance.}
\par 
 We reduce the model of the size $100$ to the size $6$. Hence, the linear systems that are required to be solved are of the size $600 \times 600$. As above, we use an unpreconditioned BiCG here. \myworries{First, let us look at the remaining assumptions for backward stability of BIRKA (see Theorem \ref{theorem:SecondCondition} and Corollary \ref{corollary:BackwardStability}). $\widehat{Q}$ is invertible here. We also have $ \left\| \widehat{Q}^{-1}\right \|$ less than one (i.e., $5.2893 \times 10^{-4}$). Finally, $\left\| \widehat{\widehat{F}}\right\|$, at the end of the first BIRKA step, for the BiCG stopping tolerance of $10^{-4}$ and $10^{-8}$ is $1.3370 \times 10^{-1}$ and $3.4528 \times 10^{-5}$, respectively, both of which are also less than one. These values are less than one at the end of all the other BIRKA steps as well.  The condition number for our problem, as defined in  \eqref{eq:25}, is $2.6653 \times 10^{-2}$. This shows that the heat transfer model is well-conditioned.}
 \par 
 For this model size, we do not give results for supporting the main conjecture (as discussed at the end of Section \ref{Sec:Analysis}; the more accurately we solve the linear systems, the more accurate reduced system we obtain). This is because for a small sized dynamical system we have already reported the data in Section \ref{Subsec:FlowModel}, and we get the similar results here. Here, we do some other analyses corresponding to Theorem \ref{theorem6}, i.e., relation between the perturbation and the stopping tolerances.
	\begin{table}[h]
		\centering
		\captionsetup{font=small}
		\begin{tabular}{|c|c|c|c|c|}
			\hline
			\begin{tabular}[c]{@{}c@{}}BIRKA\\ Iteration\end{tabular} 
			& $\left\|R_B\right\|$       &    $\left\|R_C\right\|$                    & \begin{tabular}[c]{@{}c@{}}$\left \| \left ( \widetilde{W}^T \widetilde{V} \right )^{-1}\widetilde{W}^T\right \|_F$ \\ or\\  $\left \|\widetilde{V} \left ( \widetilde{W}^T \widetilde{V} \right )^{-1} \right \|_F$\end{tabular}                      & $\left\|F\right\|$                    \\ \hline   1  & 0.0544 & $7.7746 \times 10^{-8}$ & 2.4554 & 0.1337 \\ \hline
			2  & 0.0937 & $1.2331 \times 10^{-7}$ & 2.4526 & 0.2299 \\ \hline
			3  & 0.1223 & $1.4124 \times 10^{-7}$ & 2.4515 & 0.2997 \\ \hline
			4  & 0.0568 & $9.8639 \times 10^{-8}$ & 2.4510 & 0.1392 \\ \hline
			5  & 0.0286 & $4.7669 \times 10^{-8}$ & 2.4508 & 0.0702 \\ \hline
			6  & 0.0319 & $5.2856 \times 10^{-8}$ & 2.4507 & 0.0781 \\ \hline
			7  & 0.0325 & $5.7300 \times 10^{-8}$ & 2.4507 & 0.0797 \\ \hline
			8  & 0.0325 & $6.0807 \times 10^{-8}$ & 2.4507 & 0.0796 \\ \hline
			9  & 0.0325 & $6.3895 \times 10^{-8}$ & 2.4507 & 0.0797 \\ \hline
			10 & 0.0327 & $6.6521 \times 10^{-8}$ & 2.4507 & 0.0801 \\ \hline
			11 & 0.0330 & $6.9071 \times 10^{-8}$ & 2.4507 & 0.0808      \\ \hline
		\end{tabular}
		\caption{The perturbation expression quantities for the BiCG stopping tolerance $10^{-4}$.}
		\label{pertubBiCGbig}
	\end{table}
\par 
Table \ref{pertubBiCGbig} lists the values of $\left \| R_B \right \|, \ \left \| R_C \right \|$, $\left \| \left ( \widetilde{W}^T \widetilde{V} \right )^{-1}\widetilde{W}^T\right \|$, $  \left \|\widetilde{V} \left ( \widetilde{W}^T \widetilde{V} \right )^{-1} \right \|$ and $\left \| F \right \|$ for the BiCG stopping tolerance $10^{-4}$, and Table \ref{pertubBiCGsmall} gives the same data for the BiCG stopping tolerance $10^{-8}$. It is obvious from these two tables that $\left \| \left ( \widetilde{W}^T \widetilde{V} \right )^{-1}\widetilde{W}^T\right \|$ and 
$\left \|\widetilde{V} \left ( \widetilde{W}^T \widetilde{V} \right )^{-1} \right \|$ are very less sensitive to the BiCG stopping tolerance, while $\left \| R_B \right \|$ and $\left \| R_C \right \|$ are directly proportional to it. Thus, as conjectured at the end of Section \ref{Sec:Analysis}, the norm of the perturbation $\Big(\left \| F \right \|\Big)$ should reduce as we reduce the BiCG stopping tolerance. This is supported by the data in the two tables as well (see columns for $\left \| F \right \|$). The values of $\left \| R_B \right \|$, which is the residual of the linear systems involving $\widetilde{V}$, for both the BiCG stopping tolerances seem higher that their respective stopping tolerances. The reason for this apparent anomaly is that we are reporting the absolute residuals here. The relative residuals are still less than the respective stopping tolerances. 
	\begin{table}[!]
		\centering
		\captionsetup{font=small}
		\begin{tabular}{|c|c|c|c|c|}
			\hline
			\begin{tabular}[c]{@{}c@{}}BIRKA\\ Iteration\end{tabular}
			& $\left\|R_B\right\|$       &    $\left\|R_C\right\|$                    & \begin{tabular}[c]{@{}c@{}}$\left \| \left ( \widetilde{W}^T \widetilde{V} \right )^{-1}\widetilde{W}^T\right \|_F$ \\ or\\  $\left \|\widetilde{V} \left ( \widetilde{W}^T \widetilde{V} \right )^{-1} \right \|_F$\end{tabular}                      & $\left\|F\right\|$                    \\ \hline
			1  & $1.4062 \times 10^{-5}$ & $1.3372 \times 10^{-11}$ & 2.4554 & $3.4528 \times 10^{-5}$ \\ \hline
			2  & $6.4701 \times 10^{-6}$ & $1.1488 \times 10^{-11}$ & 2.4526 & $1.5868 \times 10^{-5}$ \\ \hline
			3  & $7.3663 \times 10^{-6}$ & $9.9444 \times 10^{-12}$ & 2.4515 & $1.8058 \times 10^{-5}$ \\ \hline
			4  & $1.1982 \times 10^{-5}$ & $1.6620 \times 10^{-11}$ & 2.4510 & $2.9369 \times 10^{-5}$ \\ \hline
			5  & $9.0962 \times 10^{-6}$ & $1.1775 \times 10^{-11}$ & 2.4508 & $2.2293 \times 10^{-5}$ \\ \hline
			6  & $4.1159 \times 10^{-6}$ & $6.3212 \times 10^{-12}$ & 2.4507 & $1.0087 \times 10^{-5}$ \\ \hline
			7  & $5.2442 \times 10^{-6}$ & $8.2256 \times 10^{-12}$ & 2.4507 & $1.2852 \times 10^{-5}$ \\ \hline
			8  & $1.2491 \times 10^{-5}$ & $1.6984 \times 10^{-11}$ & 2.4507 & $3.0612 \times 10^{-5}$ \\ \hline
			9  & $1.4070 \times 10^{-5}$ & $3.6218 \times 10^{-11}$ & 2.4507 & $3.4481 \times 10^{-5}$ \\ \hline
			10 & $1.1009 \times 10^{-5}$ & $2.7919 \times 10^{-11}$ & 2.4507 & $2.6981 \times 10^{-5}$ \\ \hline
			11 & $9.4640 \times 10^{-6}$ & $2.3366 \times 10^{-11}$ & 2.4507 & $2.3193 \times 10^{-5}$  \\ \hline
		\end{tabular}
		\caption{The perturbation expression quantities for the BiCG stopping tolerance $10^{-8}$.}
		\label{pertubBiCGsmall}
	\end{table}
\par 
\myworries{We also do the sensitivity analysis of $\left \| (\widetilde{W}^T\widetilde{V})^{-1} \widetilde{W}^T \right \|$ and $ \left \| \widetilde{V}(\widetilde{W}^T\widetilde{V})^{-1} \right \|$ with respect to different random initializations of BIRKA as well as different reduced system sizes. Table \ref{sensitivity} gives this data at convergence of BIRKA corresponding to the BiCG stopping tolerance of $10^{-4}$. As evident from this table, $\left \| (\widetilde{W}^T\widetilde{V})^{-1} \widetilde{W}^T \right \|$ and $ \left \| \widetilde{V}(\widetilde{W}^T\widetilde{V})^{-1} \right \|$ vary very less. }
	\begin{table}[h]
		\centering
		\captionsetup{font=small}
		\begin{tabular}{|c|c|c|c|c|c|}
			\hline
			\multirow{2}{*}{\begin{tabular}[c]{@{}c@{}}Reduced \\ Model \\ Size\end{tabular}} & \multicolumn{5}{c|}{\begin{tabular}[c]{@{}c@{}}$\left \| \left ( \widetilde{W}^T \widetilde{V} \right )^{-1}\widetilde{W}^T\right \|_F$ \\ or\\ $\left \|\widetilde{V} \left ( \widetilde{W}^T \widetilde{V} \right )^{-1} \right \|_F$\end{tabular}}                                                                                                                                                                                         \\ \cline{2-6} 
			& \begin{tabular}[c]{@{}c@{}}Random\\  Initialization\\  1\end{tabular} & \begin{tabular}[c]{@{}c@{}}Random\\  Initialization \\ 2\end{tabular} & \begin{tabular}[c]{@{}c@{}}Random \\ Initialization \\ 3\end{tabular} & \begin{tabular}[c]{@{}c@{}}Random\\  Initialization \\ 4\end{tabular} & \begin{tabular}[c]{@{}c@{}}Random \\ Initialization \\ 5\end{tabular}  \\ \hline
			4                                                                                 & 2.0109                                                                & 2.0045                                                                & 2.0048                                                                & 2.0100                                                                & 2.0065                                                                                                                               \\ \hline
			5                                                                                 & 2.2427                                                                & 2.2406                                                                & 2.2413                                                                & 2.2399                                                                & 2.2392                                                                                                                                \\ \hline
			6                                                                                 & 2.4507                                                                & 2.4531                                                                & 2.4511                                                                & 2.4557                                                                & 2.4507                                                                                                                                \\ \hline
			7                                                                                 & 2.6467                                                                & 2.6467                                                                & 2.6468                                                                & 2.6467                                                                & 2.6467                                                                                                                               \\ \hline
			8                                                                                 & 2.8365                                                                & 2.8360                                                                & 2.8366                                                                & 2.8371                                                                & 2.8368                                                                                                                               \\ \hline
			9                                                                                 & 3.0248                                                                & 3.0269                                                                & 3.0193                                                                & 3.0306                                                                & 3.0722                                                                                                                               \\ \hline
			10                                                                                & 3.1718                                                                & 3.1759                                                                & 3.1768                                                                & 3.1711                                                                & 3.2142                                                                                                                              \\ \hline
		\end{tabular}
		\caption{The sensitivity analysis for the heat transfer model of size 100 with respect to random initializations and reduced system sizes. }
		\label{sensitivity}
	\end{table} 
\par 
We reduce the model sizes $10,000$ and $40,000$ to the sizes $6$ and $5$, respectively. Hence, the linear systems of size $60,000 \times 60,000$ and $2,00,000 \times 2,00,000$ are required to be solved, respectively. The linear systems arising in the model reduction process of both these size are ill-conditioned. \myworries{Hence, we use a preconditioned BiCG here. The preconditioner that we use is incomplete LU \cite{chow1998approximate}. The drop tolerance in the preconditioner is taken as $10^{-5}$ based upon the range given in \cite{chow1998approximate}. The result for the model size $10,000$ is given in Figure \ref{fig5} and the result for the model size $40,000$ is given in Figure \ref{fig6}}. From both Figure \ref{fig5} and \ref{fig6}, it is again evident that we get a more accurate reduced model as we solve the linear systems more accurately (solid line is below the dotted one at all the BIRKA steps). 


\begin{figure*}[h]
	\centering
	\captionsetup{font=small}
	\includegraphics[width=1\textwidth]{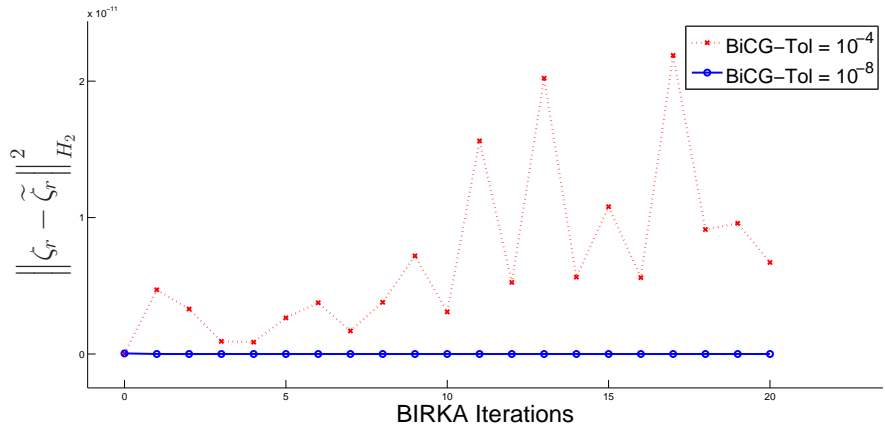}
	\caption{Accuracy of the reduced system plotted at each BIRKA iteration for the two different stopping tolerances in BiCG; heat transfer model of size 10,000.}
	\label{fig5}     
\end{figure*}

\begin{figure*}[h]
	\centering
	\captionsetup{font=small}
	\includegraphics[width=1\textwidth]{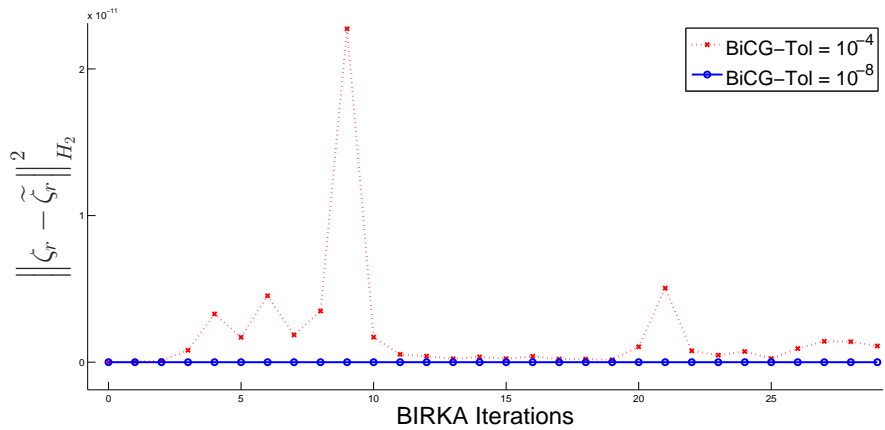}
	\caption{Accuracy of the reduced system plotted at each BIRKA iteration for the two different stopping tolerances in BiCG; heat transfer model of size 40,000.}
	\label{fig6}     
\end{figure*}

\section{Conclusions}\label{Conclusions}
\par 
BIRKA provides a locally $H_2 -$optimal reduced model. The most expensive part of BIRKA is finding solutions of large linear systems of equations. Iterative algorithms are a method of choice for such systems but they find solutions only up to a certain tolerance. Hence, we show that BIRKA is backward stable with respect to these inexact linear solves under some mild assumptions. \myworries{We also analyze the accuracy of the inexact reduced system obtained from a backward stable BIRKA.} We support all our results with numerical experiments.
\par 
The first assumption is that $\widehat{Q}$ is invertible. In Section \ref{Subec:Invetible}, we have given a better characterization of this invertibility assumption (in terms of the underlying Lyapunov equation). However, this requires further analysis.
\par 
The second and the third assumptions involve bounding $ \left\|\widehat{Q}^{-1}\right \|$ and $\left \| \widehat{\widehat{F}} \right\|  $  by one. Although for both our experimental models we have shown that these assumptions are easily satisfied, they may not always hold. $\widehat{Q}$ is dependent on the input dynamical system and $\widehat{\widehat{F}}$ on the stopping tolerance of our underlying linear solver. Hence, the future work here involves identifying the categories of bilinear dynamical systems and the range of linear solver stopping tolerances when these would be true. It is important to emphasize that the above assumptions are the sufficiency conditions (and not necessary) for backward stability of BIRKA, and hence, BIRKA may be backward stable even when these do not hold. 
\par 
\myworries{While computing the accuracy, we have given an expression for the condition number of the bilinear system with respect to computing the $H_2-$norm of the error between the perturbed model and the original model. This condition number is an {\it approximation} to the condition we want to compute. That is, the condition number of the bilinear system with respect to computing the $H_2-$norm of the error between the inexact reduced model and the original model. This is also part of future work.  } 
\par 
In literature \cite{Dirac5}, another cheaper variant of BIRKA, called Truncated BIRKA has been proposed (also called TBIRKA). TBIRKA uses a truncated Volterra series, and hence, it is computationally less expensive than BIRKA. Another future direction involves proving backward stability of TBIRKA. 
\section*{Acknowledgement}
We would like to thank Prof. Peter Benner (at Max Planck Institute
for Dynamics of Complex Technical Systems, Germany), Prof. Serkan Gugercin,  Prof. Eric de Sturler and Prof. Christopher A. Beattie  (all at Virginia Tech, USA) for discussions regarding the different aspects of this project. We would like to thank Dr. Tobias Breiten (at University of Graz, Austria) for help in understanding BIRKA code.
\par 
Thanks to the anonymous reviewers that helped to greatly improve the quality of this manuscript. We would also like to thank the editor handling our manuscript, Prof. Shmuel Friedland (at University of Illinois, USA), in giving us the flexibility during revision submissions.  
\section*{References}

\bibliography{BackBIRKA}

\begin{thebibliography}{10}
\expandafter\ifx\csname url\endcsname\relax
  \def\url#1{\texttt{#1}}\fi
\expandafter\ifx\csname urlprefix\endcsname\relax\def\urlprefix{URL }\fi
\expandafter\ifx\csname href\endcsname\relax
  \def\href#1#2{#2} \def\path#1{#1}\fi

\bibitem{Dirac3}
P.~Benner, T.~Breiten, Interpolation-based $\mathcal{H}_2$-model reduction of
  bilinear control systems, SIAM Journal on Matrix Analysis and Applications
  33~(3) (2012) 859--885.

\bibitem{Dirac5}
G.~M. Flagg, Interpolation methods for the model reduction of bilinear systems,
  Ph.D. thesis, Virginia Polytechnic Institute and State University,
  Blacksburg, VA, USA (2012).

\bibitem{Dirac10}
J.~R. Philips, Projection-based approaches for model reduction of weakly
  nonlinear, time-varying systems, IEEE Transactions on Computer-Aided Design
  of Integrated Circuits and Systems 22~(02) (2003) 171--187.

\bibitem{Dirac6}
A.~C. Antoulas, Approximation of Large-Scale Dynamical Systems, SIAM Advances
  in Design and Control, Philadelphia, PA, USA, 2005.

\bibitem{Dirac4}
S.~Gugercin, A.~C. Antoulas, C.~Beattie, $\mathcal{H}_2$ model reduction for
  large-scale linear dynamical systems, SIAM Journal on Matrix Analysis and
  Applications 30~(2) (2008) 609--638.

\bibitem{serkphd}
S.~Gugercin, Projection methods for model reduction of large-scale dynamical
  systems, Ph.D. thesis, ECE Dept., Rice University, Houston, TX, USA (2002).

\bibitem{zbai}
Z.~Bai, D.~Skoogh, A projection method for model reduction of bilinear
  dynamical systems, Linear Algebra and its Applications 415~(2-3) (2006)
  406--–425.

\bibitem{kahujaphd}
K.~Ahuja, Recycling {K}rylov subspaces and preconditioners, Ph.D. thesis,
  Virginia Polytechnic Institute and State University, Blacksburg, VA, USA
  (2011).

\bibitem{Dirac7}
E.~J. Grimme, Krylov projection methods for model reduction, Ph.D. thesis,
  University of Illinois at Urbana-Champaign, Urbana, IL, USA (1997).

\bibitem{inexirka}
C.~A. Beattie, S.~Gugercin, Inexact solves in {K}rylov-based model reduction,
  in: Proceedings of the 45th IEEE Conference on Decision and Control, 2006,
  pp. 3405--3411.

\bibitem{mimo}
A.~Bunse-Gerstner, D.~Kubali$\acute{\text{n}}$ska, G.~Vossen, D.~Wilczek,
  $h_2$-norm optimal model reduction for large scale discrete dynamical {MIMO}
  systems, Journal of Computational and Applied Mathematics 233~(5) (2010)
  1202--1216.

\bibitem{doi:10.1137/130947830}
G.~M. Flagg, S.~Gugercin, {M}ultipoint {V}olterra series interpolation and
  $\mathcal{H}_2$ optimal model reduction of bilinear systems, SIAM Journal on
  Matrix Analysis and Applications 36~(2) (2015) 549--579.

\bibitem{tobiasphd}
T.~Breiten, Interpolatory methods for model reduction of large-scale dynamical
  systems, Ph.D. thesis, Otto-von-Guericke University, Magdeburg, Germany
  (2013).

\bibitem{ahuja2012recycling}
K.~Ahuja, E.~de~Sturler, S.~Gugercin, E.~R. Chang, Recycling {B}i{CG} with an
  application to model reduction, SIAM Journal on Scientific Computing 34~(4)
  (2012) A1925--A1949.

\bibitem{sarahsms}
S.~Wyatt, Inexact solves in interpolatory model reduction, Master's thesis,
  Virginia Polytechnic Institute and State University, Blacksburg, VA, USA
  (2009).

\bibitem{sarahsphd}
S.~Wyatt, Issues in interpolatory model reduction: Inexact solves, second-order
  systems and {DAE}s, Ph.D. thesis, Virginia Polytechnic Institute and State
  University, Blacksburg, VA, USA (2012).

\bibitem{Dirac2}
C.~Beattie, S.~Gugercin, S.~Wyatt, Inexact solves in interpolatory model
  reduction, Linear Algebra and its Applications 436~(8) (2012) 2916--2943.

\bibitem{flagg12Convergence}
G.~Flagg, C.~Beattie, S.~Gugercin, Convergence of the iterative rational
  {K}rylov algorithm, Systems \& Control Letters 61~(6) (2012) 688 -- 691.

\bibitem{opac-b1090598}
W.~J. Rugh, {N}onlinear {S}ystem {T}heory: {T}he {V}olterra/{W}iener
  {A}pproach, Johns Hopkins Series in Information Sciences and Systems, Johns
  Hopkins University Press, Baltimore, 1981.

\bibitem{Dirac8}
L.~N. Trefethen, D.~Bau, Numerical Linear Algebra, SIAM, Philadelphia, PA, USA,
  1997.

\bibitem{luo2015optimization}
L.~Luo, J.~R. Edwards, H.~Luo, F.~Mueller, W.-c. Feng, Optimization of a
  fine-grained {BILU} by {CUDA} inter-block synchronization, in: Proceedings of
  the 22nd AIAA Computational Fluid Dynamics Conference, 2015, pp. 1--17
  (article number 3055).

\bibitem{van2003iterative}
H.~A. Van~der Vorst, {I}terative {K}rylov {M}ethods for {L}arge {L}inear
  {S}ystems, Vol.~13, Cambridge University Press, 2003.

\bibitem{golub2012matrix}
G.~H. Golub, C.~F. Van~Loan, {M}atrix {C}omputations, Vol.~3, Johns Hopkins
  University Press, 2012.

\bibitem{matrixanalysis}
A.~J. Laub, {M}atrix {A}nalysis for {S}cientists And {E}ngineers, SIAM,
  Philadelphia, PA, USA, 2004.

\bibitem{normkron}
P.~Lancaster, H.~K. Farahat, {N}orms on direct sums and tensor products,
  Mathematics of Computation 26~(118) (1972) 401--414.

\bibitem{meyerBook}
C.~D. Meyer, Matrix Analysis and Applied Linear Algebra, SIAM, Philadelphia,
  PA, USA, 2000.

\bibitem{heattransfermodel}
P.~Benner, T.~Damm, Lyapunov equations, energy functionals, and model order
  reduction of bilinear and stochastic systems, SIAM Journal on Control and
  Optimization 49~(2) (2011) 686--711.

\bibitem{breiten2010krylov}
T.~Breiten, T.~Damm, Krylov subspace methods for model order reduction of
  bilinear control systems, Systems \& Control Letters 59~(8) (2010) 443--450.

\bibitem{saad2003iterative}
Y.~Saad, Iterative Methods for Sparse Linear Systems, 2nd Edition, SIAM,
  Philadelphia, PA, USA, 2003.

\bibitem{chow1998approximate}
E.~Chow, Y.~Saad, Approximate inverse preconditioners via sparse-sparse
  iterations, SIAM Journal on Scientific Computing 19~(3) (1998) 995--1023.

\end{thebibliography}

\end{document}